\documentclass{amsart}

\usepackage[utf8]{inputenc}
\usepackage{amssymb} 
\usepackage{amsmath} 
\usepackage{amscd}
\usepackage{amsbsy}
\usepackage{comment}
\usepackage[matrix,arrow]{xy}
\usepackage{hyperref}
\usepackage{float}
\usepackage[utf8]{inputenc}
\usepackage{xcolor}
\usepackage{enumerate}
\usepackage{breqn}

\DeclareMathOperator{\Norm}{Norm}

\DeclareMathOperator{\Rad}{Rad}

\DeclareMathOperator{\Gal}{Gal}

\newcommand{\Q}{{\mathbb Q}}
\newcommand{\Z}{{\mathbb Z}}

\newcommand{\R}{{\mathbb R}}
\newcommand{\F}{{\mathbb F}}

\newcommand{\cO}{\mathcal{O}}

\def\mod#1{{\ifmmode\text{\rm\ (mod~$#1$)}
\else\discretionary{}{}{\hbox{ }}\rm(mod~$#1$)\fi}}

\begin{document}

\newtheorem{theorem}{Theorem}
\newtheorem{lemma}{Lemma}[section]
\newtheorem{proposition}[lemma]{Proposition}
\newtheorem{algorithm}[lemma]{Algorithm}
\newtheorem{corollary}[lemma]{Corollary}
\newtheorem*{conjecture}{Conjecture}

\theoremstyle{definition}
\newtheorem{definition}[theorem]{Definition}

\theoremstyle{remark}
\newtheorem{remark}[theorem]{Remark}

\newtheorem{acknowledgment}{Acknowledgement}

\title[]{On differences of perfect powers and prime powers}

\author{Pedro-Jos\'{e} Cazorla Garc\'{i}a}
\address{Department of Mathematics, University of Manchester, Manchester, United Kingdom, M13 9PY}
\email{pedro-jose.cazorlagarcia@manchester.ac.uk}

\date{\today}

\begin {abstract}
Given a prime number $q$ and a squarefree integer $C_1$, we develop a method to explicitly determine the tuples $(y, n, \alpha)$ for which the difference $y^n-q^\alpha$ has squarefree part equal to $C_1$. Our techniques include the combination of the local information provided by Galois representations of Frey--Hellegouarch curves with the e\-ffec\-tive resolution of Thue--Mahler equations, as well as the use of improved lower bounds for $q-$adic and complex logarithms. As an application of this methodology, we will completely resolve the case when $1 \le C_1 \le 20$ and $ 2 \le q < 25$.
\end {abstract}


\maketitle

\section{Introduction}

\subsection{Historical background}
A very famous conjecture by Catalan \cite{Catalan}, stated in 1844, asserts that the only non-zero consecutive integer perfect powers are $8$ and $9$. In terms of Diophantine equations, this is equivalent to claiming that the Diophantine equation 
\begin{equation}
\label{eqn:catalan}
y^n - x^m = 1, \quad x, y, n, m \in \Z,\quad x, y > 0, \quad n, m \ge 2,
\end{equation}
only has $(x, y, n, m) = (2, 3, 2, 3)$ as a solution. 
Mih\u{a}ilescu \cite{Mihailescu} proved Catalan's conjecture in 2004, using an argument based on the theory of cyclotomic fields and Galois modules.

Even prior to Mih\u{a}ilescu's proof of Catalan's conjecture, many researchers considered generalisations of \eqref{eqn:catalan}. For instance, 
Pillai \cite{Pillai} conjectured that, for any value of $c > 0$, the Diophantine equation
\begin{equation}
\label{eqn:generalisedcatalan}
y^n - x^m = c, \quad x, y, n, m \in \Z, \quad x, y > 0, \quad n, m \ge 2,
\end{equation}
has only finitely many solutions provided that $(n, m) \neq (2, 2)$. 
To date, Pillai's conjecture remains an open problem, and, to the best of our knowledge, there are no results unless at least one of $x$, $y$, $n$ or $m$ is fixed.

If both $x = a > 0$ and $y = b > 0$ are fixed, Bennett \cite{basefixed} showed that there are at most two solutions to \eqref{eqn:generalisedcatalan} provided that $a, b \ge 2$. This result built upon work by Pillai himself (\cite{Pillai2}) and Herschfeld \cite{Herschfeld}, and has since been generalised by Scott and Styer \cite{ScottStyer} to allow for $x$ and $y$ to be negative. 

If neither $x$ nor $y$ is fixed, much of the existing work is linked to the study of the \emph{Lebesgue--Nagell equation}
\begin{equation}
    \label{eqn:lnequation}
x^2 + D = y^n,
\end{equation}
where $D > 0$ is a fixed integer. We note that \eqref{eqn:lnequation} corresponds to the case $m=2$ in \eqref{eqn:generalisedcatalan}. In addition, we set $c = D$ for historical reasons. We refer the reader to \cite[Section 3]{survey} for a detailed exposition on the history of the Lebesgue--Nagell equation and its generalisations.

The Lebesgue--Nagell equation has been an active research topic since its first appearance in a paper by Lebesgue \cite{Lebesgue} in 1850. Indeeed, we highlight the contributions of Nagell \cite{Nagell, Nagell2}, Cohn \cite{Cohn1, Cohn2}, Mignotte and de Weger \cite{MignotteWeger} and Bennett and Skinner \cite{BenS}, which allowed for a complete resolution of \eqref{eqn:lnequation} for $D$ in the range $1 \le D \le 100$ for all but $19$ values. 

Amongst the techniques used by these researchers, two will be especially relevant for our work: the theorem on primitive divisors of Lucas--Lehmer sequences by Bilu, Hanrot and Voutier \cite{BHV} and the modular approach based on Galois re\-pre\-sen\-ta\-tions of Frey--Hellegouarch curves and modular forms, developed by Wiles, Breuil, Conrad, Diamond and Taylor \cite{modularity, TW, Wiles}.

The resolution of \eqref{eqn:lnequation} in the range $1 \le D\le 100$ was finally completed by Bugeaud, Mignotte and Siksek \cite{BMS}, who dealt with the outstanding $19$ cases by using an approach combining the aforementioned modular methodology with lower bounds on linear forms in complex logarithms, based upon Baker's theory \cite{Baker1, Baker2, Baker3}.

Very recently, the development of a new Thue--Mahler equation solver by Gherga and Siksek \cite{ThueMahler} has allowed Bennett and Siksek to improve on this combined methodology and study two cases of \eqref{eqn:lnequation} which we find particularly interesting. In \cite{BennettSiksek2}, they consider the equation 
\[x^2 + 2^{\alpha_2}3^{\alpha_3}5^{\alpha_5}7^{\alpha_7}11^{\alpha_{11}} = y^n, \quad x, y > 0, \quad \alpha_i \ge 0, \quad n \ge 3,\]
to be solved for $x, y, n, \alpha_2, \alpha_3, \alpha_5, \alpha_7$ and $\alpha_{11}$. Note that the resolution of this equation completely determines which integers are differences of a perfect power and a square, while being furthermore supported only on the primes $2, 3, 5, 7$ and $11$. With a similar set of techniques, in \cite{BennettSiksek} the same authors study the Diophantine equation
\begin{equation}
    \label{eqn:lnprimepowers}
    x^2 + q^\alpha = y^n, \quad x, y > 0, \quad \alpha > 0, \quad n \ge 3,
\end{equation}
where $2 \le q < 100$ is a fixed prime number. The resolution of this equation completely determines which squares can be written as the difference of a perfect power and a power of $q$. 

Finally, we note that the existing literature on generalisations of the Lebesgue--Nagell equation of the form
\begin{equation}
    \label{eqn:genLN}
C_1x^2 + C_2 = y^n, \quad x, y > 0, \quad n \ge 3,
\end{equation}
where $C_1 \neq 1$ is scarce. The first relevant result in this case is due to Patel \cite{Patel}, who studied \eqref{eqn:genLN} for fixed integers $1 \le C_1 \le 10$ and $1 \le C_2 \le 80$, subject to the additional constraint that $C_1C_2 \not \equiv 7 \pmod 8$. Her methods were similar to those in \cite{Cohn2} and thus relied on the primitive divisor theorem.

In work in progress, the author and Patel \cite{genLN} removed the restriction $C_1C_2 \not \equiv 7 \pmod 8$ and achieved a complete resolution of \eqref{eqn:genLN} for all values of $C_1$ and $C_2$ in the range $1 \le C_1 \le 20$ and $1 \le C_2 \le 28$. If $C_1C_2 \equiv 7 \pmod 8$, the primitive divisor theorem is no longer applicable. In these instances, the authors followed an approach combining the modular methodology with bounds coming from the theory of linear forms in logarithms.

\subsection{The main result}
Our work in this paper extends \cite{BennettSiksek} by considering a ge\-ne\-ra\-li\-sa\-tion of \eqref{eqn:lnprimepowers} in the following manner. Given a positive integer $C$, we can write it as $C = C_1(C')^2$, where $C_1$ is squarefree, and consider the following Diophantine equation:
\begin{equation}
    \label{eqn:genprimepowers}
    C_1x^2 + q^\alpha = y^n, \quad x, y > 0, \quad \alpha >0, \quad n \ge 3,
\end{equation}
where $C_1$ is a squarefree integer and $q$ is a prime number, both fixed. We note that, when compared to \eqref{eqn:generalisedcatalan}, \eqref{eqn:genprimepowers} corresponds to the case $x=q$ and $m = \alpha$. We remark that, at the expense of fixing $x = q$, we can consider all values of $c$ with squarefree part $C_1$ simultaneously.

Equation \eqref{eqn:genprimepowers} will be the main object of study of the present paper. We achieve a complete resolution of \eqref{eqn:genprimepowers} in the range $1 \le C_1 \le 20$ and $2 \le q < 25$. This is the main result of the paper and can be concisely stated as follows.

\begin{theorem}
    \label{thm:main}
    Let $C_1, q$ be integers with {$1 \le C_1 \le 20$ and $2 \le q < 25$}, with $C_1$ squarefree and $q$ prime. Then, all integer solutions $(x, y, \alpha, n)$ to the equation:
    \begin{equation}
        \label{eqn:main}
        C_1x^2 + q^\alpha = y^n, \quad \gcd(C_1x, q, y) = 1, \quad x, y > 0, \quad \alpha > 0, \quad n \ge 3,
    \end{equation}
    can be obtained from Tables \ref{tab:solutions1} and \ref{tab:solutions2}.
\end{theorem}

We note that we can assume that either $n=4$ or $n=p$ is an odd prime in \eqref{eqn:main}, and so Tables \ref{tab:solutions1} and \ref{tab:solutions2} will only include these solutions. Solutions to \eqref{eqn:main} with composite $n$ can then be easily read from those tables. Finally, we note that all solutions with $C_1 = 1$ were previously found in \cite{BennettSiksek}, but we include them here for completeness.

If $n=3$ or $n=4$, the explicit resolution of \eqref{eqn:genprimepowers} can be reduced to the determination of $S-$integral points on certain elliptic curves. If $n \ge 5$ and $y$ is odd, we may adapt the techniques developed in \cite{Patel}, which make use of the theorem on primitive divisors of Lehmer sequences, to bound $n$. For each outstanding value of $n$, we can then reduce the resolution of \eqref{eqn:genprimepowers} to solving a finite number of Thue or Thue--Mahler equations.

After that, we need to deal with the much harder case of \eqref{eqn:genprimepowers} where $y$ is even and $n \ge 5$. In this situation, the primitive divisor theorem is no longer applicable, and we will employ an approach combining the resolution of Thue--Mahler equations with the local information provided by the modular method, as well as bounds on $n$ coming from the results of Bugaud and Laurent \cite{qadiclogs} on lower bounds for linear forms in $q-$adic logarithms and from the newly improved lower bounds on linear forms in complex logarithms, developed by Mignotte and Voutier (\cite{Voutier}).

\subsection{Comparison with previously existing literature}
We now highlight the most relevant innovations in this paper with regards to the previously existing literature. In the case where $y$ is odd, we adapt the methodology in \cite{Patel} to the case where $C_2 = q^\alpha$, while also introducing several computational improvements. In many cases, and as we shall see in Section \ref{Sec:yodd}, this allows us to bypass the resolution of Thue--Mahler equations completely, with very relevant computational savings.

In order to obtain a bound for the exponent $n$ in \eqref{eqn:main}, we develop two methods which are successful in some situations where classical modular method techniques (e.g. Proposition \ref{prop:dirbound}) fail. These methods are based on similar techniques in work in progress by the author and Patel \cite{genLN}. Firstly, we define a new Frey--Hellegouarch curve as the quadratic twist of our original curve by an integer $\ell \mid C_1$, and use it to employ a multi-Frey approach. We present this in Section \ref{Sec:QTs}. Secondly, given a prime number $\ell$, we use basic Galois theory to determine conditions under which the reduction of the Frey--Hellegouarch curve will have full $2-$torsion over $\F_\ell$. In Section \ref{Subsec:galois}, we explain how to exploit this additional structure to obtain a bound on $n$.


In Section \ref{Subsec:modularThue}, we explain a new way to combine the local information provided by the modular method with Thue--Mahler equations to prove the non-existence of solutions to \eqref{eqn:genprimepowers}. This is extremely useful in situations where Kraus's method (see Proposition \ref{prop:modifiedKraus}) fails and where the explicit resolution of Thue--Mahler equations is computationally unfeasible. 

Finally, if the modular methodology is unsuccessful in bounding $n$, we use the newly-improved lower bounds for linear forms in complex logarithms in \cite{Voutier} to bound $n$. Compared to previous applications of linear forms in logarithms to similar Diophantine equations (e.g. \cite{BMS, BennettSiksek, BennettSiksek2}), our bounds are around $50\%$ smaller, giving a substantial saving in computation time. This is presented in Section \ref{Sec:LFL}.



\subsection{Structure of the paper}
The outline of this paper is as follows. In Section \ref{Sec:smallexponents}, we will find all solutions of \eqref{eqn:main} with $n = 3$ and $n = 4$ by reducing the problem to that of finding $S-$integral points on elliptic curves. In Section \ref{Sec:yodd}, we find all solutions of \eqref{eqn:main} with $y$ odd in the range $1 \le C_1 \le 20$ and $2 \le q < 25$ by applying results derived from the theorem on primitive divisors of Lucas-Lehmer sequences (\cite{BHV}) and by refining the Thue--Mahler solver developed in \cite{ThueMahler}. In Section \ref{Sec:ThueMahleryEven}, we explain how to reduce \eqref{eqn:main} with $y$ even to a Thue--Mahler equation and we solve the cases $n=5$ and $n=7$. In Section \ref{Sec:modularmethod}, we introduce the modular method which we will use in the following two sections to prove that there are almost no remaining solutions. In Section \ref{Sec:boundexponent}, we present four techniques involving the modular method that we may use to bound the exponent $n$ in \eqref{eqn:main} for some values of $C_1$ and $q$. Then, in Section \ref{Sec:specificexponents}, we will develop some methodology to show that \eqref{eqn:main} has no solutions with $y$ even for a fixed value of $n \ge 11$. In Section \ref{Sec:LFL}, we will use the new estimates for linear forms in complex logarithms in \cite{Voutier} in order to bound $n$ and show that all solutions to \eqref{eqn:main} have been found in previous sections. Finally, in Section \ref{Sec:conclusions} we will compile all the previous results to prove Theorem \ref{thm:main}.

All the code that we have used to perform computations in this paper is pu\-blic\-ly available in \href{https://github.com/PJCazorla/Differences-between-perfect-and-prime-powers}{https://github.com/PJCazorla/Differences-between-perfect-and-prime-powers}. \\

\noindent \textbf{Acknowledgements} The author would like to thank Gareth Jones and Martin Orr for comments on a draft version of the paper, and Adela Gherga and Samir Siksek for useful discussions.

\begin{center}
\begin{table}[!ht]
\begin{tabular}[t]{||c c c c c c||}
\hline 
$C_1$ & $q$ & $x$ & $y$ & $\alpha$ & $n$ \\
\hline\hline 
1 & 2 & 5 & 3 &  1  &  3\\
\hline
1 & 2 & 7 & 3 &  5  &  4\\
\hline 
1 & 2 & 11 & 5 &  2  &  3\\
\hline 
1 & 3 & 10 & 7 &  5  &  3\\
\hline 
1 & 3 & 46 & 13 &  4  &  3\\
\hline 
1 & 7 & 1 & 2 &  1  &  3\\
\hline 
1 & 7 & 3 & 2 &  1  &  4\\
\hline 
1 & 7 & 13 & 8 &  3  &  3\\
\hline 
1 & 7 & 24 & 5 &  2  &  4\\
\hline
1 & 7 & 181 & 32 &  1  &  3\\
\hline 
1 & 7 & 524 & 65 &  2  &  3\\
\hline
1 & 7 & 5 & 2 & 1 & 5 \\
\hline 
1 & 7 & 181 & 8 & 1 & 5 \\
\hline 
1 & 7 & 11 & 2 & 1 & 7 \\
\hline 
1 & 11 & 2 & 5 &  2  &  3\\
\hline 
1 & 11 & 4 & 3 &  1  &  3 \\
\hline 
1 & 11 & 58 & 15 &  1  &  3 \\
\hline 
1 & 11 & 9324 & 443 &  3  &  3 \\
\hline 
1 & 13 & 70 & 17 &  1  &  3 \\
\hline 
1 & 17 & 8 & 3 &  1  &  4 \\
\hline 
1 & 19 & 18 & 7 &  1  &  3 \\
\hline 
1 & 19 & 22434 & 55 & 1 & 5\\
\hline 
1 & 23 & 2 & 3 & 1 & 3 \\
\hline 
1 & 23 & 588 & 71 & 3 & 3 \\
\hline 
1 & 23 & 6083 & 78 & 3 & 4 \\
\hline 
1 & 23 & 3 & 2 & 1 & 5 \\
\hline 
1 & 23 & 45 & 2 & 1 & 11 \\
\hline 
2 & 3 & 7 & 5 &  3  &  3 \\
\hline 
2 & 3 & 25 & 11 &  4  &  3 \\
\hline 
2 & 3 & 146 & 35 &  5  &  3 \\
\hline 
2 & 3 & 21395 & 971 &  8  &  3\\
\hline 
2 & 5 & 1 & 3 &  2  &  3 \\
\hline 
2 & 5 & 13 & 7 &  1  &  3\\
\hline 
2 & 5 & 134 & 33 &  2  &  3\\
\hline
2 & 7 & 4 & 3 &  2  &  4\\
\hline 
2 & 7 & 19 & 9 &  1  &  3 \\
\hline 
2 & 7 & 128060 & 3201 &  4  &  3 \\
\hline 
2 & 13 & 41 & 15 &  1  &  3 \\
\hline 
2 & 13 & 68 & 21 &  1  &  3 \\
\hline 
2 & 13 & 804 & 109 &  3  &  3 \\
\hline 
\end{tabular}
\begin{tabular}[t]{||c c c c c c||}
\hline 
$C_1$ & $q$ & $x$ & $y$ & $\alpha$ & $n$ \\
\hline \hline
2 & 13 & 75090 & 2797 &  9  &  3 \\
\hline 
2 & 17 & 56 & 9 &  2  &  4 \\
\hline 
2 & 19 & 2 & 3 &  1  &  3 \\
\hline
2 & 19 & 33 & 13 &  1  &  3 \\
\hline 
2 & 19 & 2981 & 261 &  3  &  3 \\
\hline 
2 & 19 & 1429 & 21 & 1 & 5 \\
\hline 
2 & 23 & 10 & 9 & 2 & 3 \\
\hline 
2 & 23 & 84 & 11 & 2 & 4\\
\hline 
2 & 23 & 122 & 31 & 1 & 3 \\
\hline 
3 & 2 & 21 & 11 &  3  &  3 \\
\hline 
3 & 5 & 1 & 2 &  1  &  3 \\
\hline
3 & 5 & 13 & 8 &  1  &  3 \\
\hline 
3 & 5 & 840211 & 12842 &  7  &  3 \\
\hline
3 & 5 & 3 & 2 & 1 & 5 \\
\hline
3 & 5 & 1 & 2 & 3 & 7 \\
\hline 
3 & 13 & 1 & 2 &  1  &  4 \\
\hline 
3 & 13 & 9 & 4 &  1  &  4 \\
\hline 
3 & 13 & 51 & 10 &  3  &  4 \\
\hline 
3 & 13 & 245 & 82 &  5  &  3 \\
\hline 
3 & 13 & 4471 & 88 &  1  &  4 \\
\hline 
3 & 17 & 6 & 5 &  1  &  3 \\
\hline 
3 & 23 & 208 & 19 & 2 & 4 \\
\hline 
5 & 2 & 43 & 21 &  4  &  3 \\
\hline 
5 & 3 & 1 & 2 &  1  &  3 \\
\hline
5 & 3 & 2596 & 323 &  7  &  3 \\
\hline 
5 & 3 & 1 & 2 & 3 & 5 \\
\hline 
5 & 3 & 5 & 2 & 1 & 7 \\
\hline 
5 & 3 & 19 & 2 & 5 & 11 \\
\hline
5 & 7 & 2 & 3 &  1  &  3 \\
\hline 
5 & 11 & 1 & 2 &  1  &  4 \\
\hline 
5 & 11 & 7 & 4 &  1  &  4 \\
\hline 
5 & 11 & 57 & 26 &  3  &  3 \\
\hline 
5 & 19 & 3 & 4 &  1  &  3 \\
\hline 
5 & 19 & 14423 & 1014 &  5  &  3 \\
\hline 
5 & 23 & 8 & 7 & 1 & 3 \\
\hline 
6 & 11 & 19 & 7 & 4 & 5 \\
\hline 
6 & 17 & 3 & 7 &  2  &  3 \\
\hline 
6 & 17 & 45 & 23 &  1  &  3 \\
\hline 
6 & 17 & 3084 & 385 &  2  &  3 \\
\hline 
6 & 19 & 51 & 25 &  1  &  3 \\
\hline 
\end{tabular}
\caption{First part of solutions to \eqref{eqn:main} with $n$ prime or $n = 4$, {$1 \le C_1 \le 20$ and $2 \le  q < 25$}, with $C_1$ squarefree, $q$ prime, $x, y > 0$, $\alpha > 0$ and $\gcd(C_1x, q, y) = 1$. }
\label{tab:solutions1}
\end{table}
\end{center}

\begin{center}
    \begin{table}[!ht]
\begin{tabular}[t]{||c c c c c c ||}
\hline 
$C_1$ & $q$ & $x$ & $y$ & $\alpha$ & $n$ \\
\hline \hline 

6 & 23 & 4 & 5 & 2 & 4 \\
\hline 
6 & 23 & 105378 & 4057 & 6 & 3 \\
\hline 
6 & 23 & 843570 & 16223 & 3 & 3 \\
\hline 
7 & 3 & 1 & 2 &  2  &  4 \\
\hline 
7 & 3 & 5 & 4 &  4  &  4 \\
\hline 
7 & 3 & 38 & 31 &  9  &  3 \\
\hline 
7 & 3 & 430 & 109 &  6  &  3 \\
\hline 
7 & 5 & 1 & 2 & 2 & 5 \\
\hline
7 & 5 & 17 & 2 & 2 & 11 \\
\hline 
7 & 11 & 1 & 2 &  2  &  7 \\
\hline 
7 & 13 & 4 & 5 &  1  &  3 \\
\hline 
7 & 13 & 7 & 8 &  2  &  3 \\
\hline 
7 & 17 & 21 & 20 &  3  &  3 \\
\hline 
10 & 3 & 1 & 13 &  7  &  3 \\
\hline 
10 & 3 & 71 & 37 &  5  &  3 \\
\hline 
10 & 11 & 474 & 131 &  3  &  3 \\
\hline 
10 & 11 & 646 & 161 &  2  &  3 \\
\hline 
10 & 17 & 1 & 3 &  1  &  3 \\
\hline 
10 & 17 & 48 & 113 &  5  &  3 \\
\hline 
11 & 2 & 1 & 3 &  4  &  3 \\
\hline 
11 & 2 & 85 & 43 &  5  &  3 \\
\hline 
11 & 3 & 2 & 5 &  4  &  3 \\
\hline 
11 & 5 & 1 & 2 &  1  &  4 \\
\hline 
11 & 5 & 8 & 9 &  2  &  3 \\
\hline 
11 & 5 & 19 & 8 &  3  &  4 \\
\hline 
11 & 5 & 19 & 16 &  3  &  3 \\
\hline 
11 & 5 & 59 & 14 &  3  &  4 \\
\hline 
11 & 7 & 1696 & 75 &  2  &  4 \\
\hline 
11 & 13 & 23 & 18 &  1  &  3 \\
\hline 
11 & 13 & 93 & 46 &  3  &  3 \\
\hline 
11 & 17 & 288 & 97 &  2  &  3 \\
\hline 
11 & 23 & 5644 & 705 & 2 & 3 \\
\hline
13 & 2 & 3 & 5 &  3  &  3 \\
\hline 
13 & 2 & 67 & 45 &  15  &  3 \\
\hline 
13 & 3 & 1 & 2 &  1  &  4 \\
\hline 
13 & 3 & 1 & 4 &  5  &  4 \\
\hline 
13 & 3 & 71 & 16 &  1  &  4 \\
\hline 
13 & 5 & 6668 & 833 &  4  &  3 \\
\hline 
13 & 7 & 342 & 115 &  3  &  3 \\
\hline 
13 & 11 & 31 & 24 &  3  &  3 \\
\hline 

\end{tabular}
\begin{tabular}[t]{|| c c c c c c||}
\hline 
$C_1$ & $q$ & $x$ & $y$ & $\alpha$ & $n$ \\
\hline \hline 
13 & 11 & 3 & 2 &  1  &  7 \\
\hline 
13 & 17 & 60 & 19 &  4  &  4 \\
\hline 
13 & 19 & 42 & 31 &  3  &  3 \\
\hline 
13 & 19 & 1 & 2 & 1 & 5 \\
\hline 
13 & 23 & 12 & 7 & 2 & 4 \\
\hline 
13 & 23 & 1032 & 61 & 2 & 4 \\
\hline 
14 & 5 & 2 & 3 &  2  &  4 \\
\hline 
14 & 11 & 6 & 5 &  2  &  4 \\
\hline 
14 & 13 & 1 & 3 &  1  &  3 \\
\hline 
14 & 13 & 902 & 225 &  2  &  3 \\
\hline 
14 & 19 & 4 & 3 & 1 & 5 \\
\hline 
15 & 7 & 1 & 4 &  2  &  3 \\
\hline 
15 & 7 & 136 & 23 &  4  &  4 \\
\hline 
15 & 7 & 33 & 4 &  2  &  7 \\
\hline 
15 & 11 & 3 & 4 &  2  &  4 \\
\hline 
15 & 13 & 2597 & 466 &  4  &  3 \\
\hline 
15 & 13 & 5124 & 733 &  3  &  3 \\
\hline 
15 & 17 & 1 & 2 & 1 & 5 \\
\hline 
15 & 17 & 7 & 4 & 2 & 5 \\
\hline 
15 & 23 & 103 & 76 & 4 & 3 \\
\hline 
17 & 2 & 1 & 3 &  6  &  4 \\
\hline
17 & 2 & 9 & 7 &  10  &  4 \\
\hline 
17 & 2 & 231 & 31 &  14  &  4 \\
\hline 
17 & 3 & 8 & 11 &  5  &  3 \\
\hline 
17 & 7 & 1375 & 318 &  5  &  3 \\
\hline 
17 & 13 & 2 & 3 &  1  &  4 \\
\hline 
17 & 13 & 6 & 5 &  1  &  4 \\
\hline 
17 & 23 & 25 & 22 & 1 & 3 \\
\hline 
19 & 2 & 1 & 3 &  3  &  3 \\
\hline 
19 & 2 & 63 & 43 &  12  &  3 \\
\hline 
19 & 2 & 4095 & 683 &  9  &  3 \\
\hline 
19 & 2 & 76539 & 4931 &  33  &  3 \\
\hline 
19 & 3 & 28 & 25 &  6  &  3 \\
\hline 
19 & 5 & 2 & 3 &  1  &  4 \\
\hline 
19 & 5 & 91 & 54 &  3  &  3 \\
\hline 
19 & 7 & 2 & 5 &  2  &  3 \\
\hline 
19 & 7 & 16 & 17 &  2  &  3 \\
\hline 
19 & 13 & 468200376 & 1608937 &  6  &  3 \\
\hline 
19 & 13 & 1 & 2 & 1 & 5\\
\hline
\end{tabular}
\caption{Second part of solutions to \eqref{eqn:main} with $n$ prime or $n = 4$, {$1 \le C_1 \le 20$ and $2 \le q < 25$}, with $C_1$ squarefree, $q$ prime, $x, y > 0$, $\alpha > 0$ and $\gcd(C_1x, q, y) = 1$. }
\label{tab:solutions2}
\end{table}
\end{center}

\section{\texorpdfstring{Small exponents: $n = 3$ and $n = 4$}{Small exponents: n = 3 and n = 4}}
\label{Sec:smallexponents}
In this section, let $S = \{q\}$. We shall explain how to solve \eqref{eqn:main} for $n = 3$ and $n = 4$ by reducing the problem to that of finding $S-$integral points on certain elliptic curves.

Let $(x, y, \alpha, n)$ be a solution to \eqref{eqn:main} with $n = 3$ and let us write $\alpha = 6k + i$, with $k \ge 0$ and $i \in \{0,1,\dots, 5\}$. Let $X = C_1y/q^{2k}$ and $Y = C_1^2x/q^{3k}$. Then, it follows that $(X, Y)$ is an $S-$integral point on the elliptic curve
\[E_{C_1, q, i}: Y^2 = X^3-C_1^3q^i.\]

Similarly, if $(x, y, \alpha, n)$ is a solution to \eqref{eqn:main} with $n = 4$, we write $\alpha = 4l+j$ with $l \ge 0$ and $j \in \{0, 1, 2, 3\}$. Let $X = C_1y^2/q^{2l}$ and $Y = C_1^2xy/q^{3l}$. Then, $(X, Y)$ is an $S-$integral point on the elliptic curve
\[F_{C_1, q, j}: Y^2 = X^3-C_1^2q^jX.\]

For each of these two cases, we may determine all $S-$integral points by using the algorithm presented in \cite{PZG}, based on lower bounds for linear forms in elliptic logarithms and upper bounds on the size of $S$-integral points of elliptic curves. We shall use the implementation of the algorithm on the computer algebra system \texttt{Magma} \cite{Magma} in order to retrieve all solutions $(x, y, \alpha, n)$ where $x, y> 0$, $\alpha > 0$, $n = 3, 4$ and $\gcd(C_1x, q, y) = 1$. This algorithm is successful in all but two cases, corresponding to the pairs $(C_1, q) = (7, 23)$ and $(C_1, q) = (19, 23)$ and the elliptic curves
\[E_{7, 23, 5}: Y^2 = X^3-2207665649,\]
and
\[E_{19, 23, 5}: Y^2 = X^3 - 44146876637.\]
In both cases, the \texttt{Magma} subroutine was able to determine that the curves had rank one but was unable to find a generator for the Mordell-Weil group. We can find such an element by computing Heegner points on the curves, following the algorithm of Gross and Zagier \cite{GZ}. We succeed in both cases and proceed to find the $S-$integral points in the same manner as in the rest of the situations.

In this way, we obtain {164} solutions to \eqref{eqn:main} with $n=3$ or $n=4$, all of which are recorded in Tables \ref{tab:solutions1} and \ref{tab:solutions2}.

\section{\texorpdfstring{The case where $y$ is odd}{The case where y is odd}}
\label{Sec:yodd}

In this section, we will solve Equation \eqref{eqn:main} under the assumption that $y$ is odd. Our main tool to bound the exponent $n$ in \eqref{eqn:main} is Theorem 1 in \cite{Patel}, which is based upon the theorem of Bilu, Hanrot and Voutier on the existence of primitive divisor of Lucas-Lehmer sequences (\cite{BHV}). Then, we shall improve upon the computational methodology in  \cite{Patel} to reduce the resolution of \eqref{eqn:main} with $y$ odd to a finite number of Thue and Thue--Mahler equations.

We can solve the former with the \texttt{Magma} in-built Thue solver and the latter with the Thue--Mahler solver developed in \cite{ThueMahler}. This will allow us to prove the following Proposition, which completely solves \eqref{eqn:main} if $y$ is odd in the range $1 \le C_1 \le 20$ and $ 2 \le q < 25$.

\begin{proposition}
    \label{prop:yodd}
    Let $(x, y, \alpha, n)$ be a solution to \eqref{eqn:main} with {$1 \le C_1 \le 20$} squarefree, {$2 \le q < 25$ prime}, $n \ge 5$, $x > 0$ and $y$ odd. Then, 
    \begin{align*}
(C_1, q, x, y, \alpha, n) & \in  \{(1, 19, 22434, 55, 1, 5), (2, 19, 1429, 21, 1, 5), \\
& (6, 11, 19, 7, 4, 5), (14, 19, 4, 3, 1, 5) \}
    \end{align*}
    and all of this solutions are included in Tables \ref{tab:solutions1} and \ref{tab:solutions2}.
\end{proposition}

Given the results in Section \ref{Sec:smallexponents}, we will assume that $n = p \ge 5$ is a prime. Then, we can bound $p$ by using the following result, which is an extension of Theorem 1 in \cite{Patel}, originally stated in \cite{genLN}. We include the proof of that result here for convenience.

\begin{theorem}
    \label{thm:generalyodd}
    Let $C_1$ be a positive squarefree integer and $C_2$ a positive integer. Write $C_1C_2 = cd^2$ where $c$ is squarefree. Let $p$ be an odd prime for which the equation 
    \begin{equation*}
        C_1x^2 + C_2 = y^p, \quad x,~y > 0, \quad \gcd(C_1x^2,C_2,y^p)=1,
    \end{equation*} 
    has a solution $(x,y)$, with either $C_1C_2 \not\equiv 7\pmod{8}$ or $C_1C_2 \equiv 7\pmod{8}$ and $y$ is odd. 
    Then either, 
    \begin{enumerate}[(i)]
    \item $p \le 5$, or
    \item $p=7$ and $y=3$, $5$ or $9$, or 
    \item $p$ divides the class number of $\Q(\sqrt{-c})$, or
    \item $p \mid \left( \ell - \left(\frac{-c}{\ell}\right)\right)$, where $\ell$ is some prime $\ell \mid d$ and $\ell\nmid 2c$.  
    \end{enumerate}
\end{theorem}

\begin{proof}
    If $C_1C_2 \not \equiv 7 \pmod 8$, the result follows by \cite[Theorem 1]{Patel}. Otherwise, we have by assumption that $C_1C_2 \equiv 7 \pmod 8$ and $y$ is odd.
    

    In this situation, we can apply the primitive divisor by Bilu, Hanrot and Voutier (BHV) in an identical manner to the proof of Theorem 1 in \cite{Patel}, as the key assumption there is precisely that $y$ is odd.
\end{proof}

We may then apply Theorem \ref{thm:generalyodd} to \eqref{eqn:main}, proving the following corollary.

\begin{corollary}
    \label{cor:yodd}
    Let $C_1 > 0$ be a squarefree integer and $q$ a prime number. Suppose that $(x, y, \alpha, n)$ is a solution to \eqref{eqn:main} with $n=p$ a prime and $y$ odd. Then, either:
    \begin{enumerate}[(a)]
        \item $p \le 5$, or
        \item $p=7$ and $y=3, 5$ or $9$, or
        \item $\alpha$ is odd and $p$ divides the class number of $\mathbb{Q}(\sqrt{-C_1q})$, or
        \item $\alpha$ is even and $p$ divides the class number of $\mathbb{Q}(\sqrt{-C_1})$, or
        \item $\alpha$ is even, $q \neq 2$ and $p \mid \left(q - \left(\frac{-C_1}{q}\right)\right)$.
    \end{enumerate}
    
\end{corollary}

\begin{proof}
    Conditions (i) and (ii) in Theorem \ref{thm:generalyodd} are identical to conditions (a) and (b) in the statement of the corollary, so suppose that none of the two hold. First, let us assume that $\alpha = 2k+1$ with $k \ge 0$ an integer. Then, in the notation of Theorem \ref{thm:generalyodd}, it follows that $C_2 = q^\alpha$ and that
    \begin{equation}
    \label{eqn:cdalphaodd}
        c = C_1q, \quad d = q^k, \quad k \ge 0.
    \end{equation}
    In addition, the only prime $\ell \mid d$ is $\ell = q$, which also divides $2c$. Consequently, condition (iii) on Theorem \ref{thm:generalyodd} necessarily holds and we get condition (c) in this corollary.
    
    Suppose now that $\alpha = 2k$ with $k > 0$ an integer. Then, we have that $C_2 = q^\alpha$ and that
    \begin{equation}
        \label{eqn:cdalphaeven}
        c = C_1, \quad d = q^k, \quad k > 0,
    \end{equation}
    and the only prime $\ell \mid d$, $\ell \nmid 2c$ is $\ell = q$, provided that $q \neq 2$. Then, conditions (iii) and (iv) in Theorem \ref{thm:generalyodd} give rise to conditions (d) and (e), finishing the proof of the corollary.
    
    
\end{proof}

We now explain how to adapt the computational methodology in Section 6 of \cite{Patel} to our case. First, let us treat the case where $p$ does not satisfy (c) or (d) in Corollary \ref{cor:yodd}. This case is summarised in the following lemma.

\begin{lemma}
    \label{lemma:thuemahleryodd}
    Suppose that $(x, y, \alpha, n)$ is a solution to \eqref{eqn:main} with 
    $n=p \ge 5$ prime and $y$ odd. Let $c, d$ be as in \eqref{eqn:cdalphaodd} if $\alpha = 2k+1$ for some $k \ge 0$ and as in \eqref{eqn:cdalphaeven} if $\alpha = 2k$ for some $k > 0$. Suppose furthermore that $p$ does not divide the class number of $\mathbb{Q}(\sqrt{-c})$, and let us define $G(U, V) \in \Z[U,V]$ by the following expression:
    \begin{equation}
        \label{eqn:nonreducedthuemahler}
        G(U, V) = \frac{(U+V\sqrt{-c})^p - (U-V\sqrt{-c})^p }{2\sqrt{-c}}.
    \end{equation}
    Then, $y$ satisfies 
    \begin{equation}
        \label{eqn:recoveryoddeasycase}
    y = \frac{r^2+cs^2}{C_1} \text{ if } -c \not \equiv 1 \pmod 4,\end{equation} or     
    \begin{equation}
        \label{eqn:recoveryoddhardcase}
    y = \frac{r^2+cs^2}{4C_1} \text{ if } -c  \equiv 1 \pmod 4.
    \end{equation}
    and $(r, s)$ satisfy the following non-reduced Thue--Mahler equation:
    \[G(r, s) = C_1^{(p-1)/2}q^k, \quad \text{if} -c \not \equiv 1 \pmod 4,\]
    or
    \[G(r, s) = 2^pC_1^{(p-1)/2}q^k, \quad \text{if} -c \equiv 1 \pmod 4.\]
    In addition, we have that $s \in S_{c, q}$, where the set $S_{c,q}$ is defined by:
    \[S_{c,q} = \begin{cases}
        \{\pm 1, \pm q^k\}, & \text{ if } -c \not \equiv 1\pmod 4, \quad q \nmid p, \\
        \{\pm 1, \pm q^{k-1}, \pm q^k\}, &\text{ if } -c \not \equiv 1 \pmod 4, \quad q \mid p, \\
        \{\pm 1, \pm 2, \pm q^k, \pm 2\cdot q^k\},& \text{ if } -c \equiv 1 \pmod 4, \quad q \nmid 2p, \\
        \{\pm 1, \pm 2, \pm q^{k-1}, \pm 2\cdot q^{k-1}, \pm q^k, \pm 2\cdot q^k\}, & \text{ if } -c \equiv 1 \pmod 4, \quad q \mid p, \quad q \neq 2,\\
        \{\pm 1, \pm 2, \dots, \pm 2^{(p-3)/2}, \pm 2^{(p-1)/2}, \pm 2^{k+1}\},& \text{ if } -c \equiv 1 \pmod 4, \quad q = 2. 
        \end{cases}
        \]
\end{lemma}

\begin{proof}
Since, by assumption, $p$ does not divide the class number of $\Q(\sqrt{-c})$, this lands into \textbf{case I} in \cite{Patel} and, as shown there, there exist integers $r, s$ satisfying \eqref{eqn:recoveryoddhardcase} or \eqref{eqn:recoveryoddeasycase}, depending on whether $-c \equiv 1 \pmod 4$ or not.
In order to find $r$ and $s$, we distinguish two cases. If $-c \not \equiv 1 \pmod 4$, we have, again by \cite{Patel}, that $s \mid q^k$ and $r$ satisfies the following equation:
\begin{equation}
    \label{eqn:auxTM1}
0 = f_s(r) = \frac{(r+s\sqrt{-c})^p - (r-s\sqrt{-c})^p }{2s\sqrt{-c}} - \frac{C_1^{({p-1})/{2}}q^k}{s}.
\end{equation}
Similarly, if $-c \equiv 1 \pmod 4$, we have that $s \mid 2q^k$ and $r$ satisfies:
\begin{equation}
    \label{eqn:auxTM2}
0 = f_s(r) = \frac{(r+s\sqrt{-c})^p - (r-s\sqrt{-c})^p }{2s\sqrt{-c}} - \frac{2^pC_1^{({p-1})/{2}}q^k}{s}.
\end{equation}
Multiplying both equalities by $s$, we obtain the non-reduced Thue--Mahler equations present in the statement of the Lemma. 

Now, let us find the expressions for $S_{c, q}$. Suppose first that $q \nmid 2p$. Then, applying the binomial theorem yields that
\begin{equation}
    \label{eqn:binomial}
\frac{(r+s\sqrt{-c})^p - (r-s\sqrt{-c})^p }{2s\sqrt{-c}} = pr^{p-1} + s^2H(r, s), 
\end{equation}
for certain polynomial $H(r, s) \in \Z[r, s]$. Suppose for contradiction that $q \mid s$ and that $q^k \nmid s$. Then, $q \mid s^2H(r, s)$ and either
\[q \mid \frac{C_1^{(p-1)/2}q^k}{s} \quad \text{ if } -c \not \equiv 1 \pmod 4\]
or 
\[q \mid \frac{2^pC_1^{(p-1)/2}q^k}{s} \quad \text{ if } -c \equiv 1 \pmod 4.\]
These expressions, together with \eqref{eqn:auxTM1}, \eqref{eqn:auxTM2} and \eqref{eqn:binomial}, yield that $q \mid pr^{p-1}$. Since $q \nmid p$, it follows that $q \mid \gcd(r, s)$. By \eqref{eqn:recoveryoddeasycase} and \eqref{eqn:recoveryoddhardcase}, along with the fact that $C_1$ is squarefree, we have that $q \mid y$, which is a contradiction with the fact that $\gcd(q, y) = 1$.

Consequently, either $q \nmid s$ or $q^k \mid s$. Since $s \mid q^k$ for $-c \not \equiv 1 \pmod 4$ and $s \mid 2q^k$ for $-c \equiv 1 \pmod 4$, we get the expressions for $S_{c, q}$ in the statement of the lemma. \\

If $q \mid p$ and $q \neq 2$, we follow an identical argument to show that either $q \nmid s$ or $q^{k-1} \mid s$. Finally, if $q = 2$, the same reasoning is valid if $-c \not \equiv 1 \pmod 4$. If $-c \equiv 1 \pmod 4$, we may adjust it to prove that either $2^{(p+1)/2} \nmid s$ or $2^{k+1} \mid s$. This gives the expressions for $S_{c, q}$ presented above, thereby concluding the proof.
\end{proof}

Now, we shall explain how to solve the non-reduced Thue--Mahler equation in each of the cases. First, we note that the expression for $G(U, V)$ in \eqref{eqn:nonreducedthuemahler} can be rewritten as:
\[G(U, V) = V\cdot F(U, V),\]
where $F(U, V) \in \Z[U, V]$ has degree $p-1$. It is sufficient to solve the equation:
\begin{equation}
    \label{eqn:thuemahleryodd1}
F(U, V) = \frac{C_1^{(p-1)/2}q^k}{s}, \quad \text{if} -c \not \equiv 1 \pmod 4,
\end{equation}
or
\begin{equation}
    \label{eqn:thuemahleryodd2}
F(U, V) = \frac{2^pC_1^{(p-1)/2}q^k}{s}, \quad \text{if} -c \equiv 1 \pmod 4,
\end{equation}
for each value of $s$ in $S_{c,q}$. If $s = \pm q^{k-1}, \pm 2 \cdot q^{k-1}, \pm q^k, \pm 2 \cdot q^k$, \eqref{eqn:thuemahleryodd1} and \eqref{eqn:thuemahleryodd2} reduce to Thue equations, since the right-hand-side of those identities no longer depends on $k$. 
We may solve these Thue equations with the \texttt{Magma} in-built Thue solver, which is based upon lower bounds on linear forms on elliptic logarithms. Given a solution $(U, V)$ of the relevant Thue equation, it is then elementary to see if there are any values of $k$ for which $V = s$ for our particular choice of $s$.

We emphasise that, in general, it is much more efficient computationally to solve Thue equations rather than Thue--Mahler equations, so this approach gives a significant improvement.

Suppose now that $s = \pm 1, \pm 2$ (or $s = \pm 1, \pm 2, \dots, \pm 2^{(p-1)/2}$ in the case $q=2$). Then, the right-hand side of \eqref{eqn:thuemahleryodd1} and \eqref{eqn:thuemahleryodd2} does depend on $k$, so the two expressions are now Thue--Mahler equations. As mentioned above, solving this is, in general, very expensive computationally, so we present a further trick that may allow us to bypass the solution of certain Thue--Mahler equations.

Indeed, let us define $f(U) = F(U, s) \in \mathbb{Z}[U]$ for our fixed value of $s$. If the polynomial $f(U)$ does not have roots in $\Z_q$, this means that there exists a constant $k_0 \ge 1$ for which the congruence equation
\[F(U, s) \equiv 0 \pmod{q^k}\]
does not have a solution for all $k \ge k_0$. 
It then follows that all solutions to \eqref{eqn:thuemahleryodd1} or \eqref{eqn:thuemahleryodd2} satisfy that $k < k_0$ and finding solutions for \eqref{eqn:thuemahleryodd1} and \eqref{eqn:thuemahleryodd2} amounts to finding roots of polynomials.

If $f(U)$ has roots in $\Z_q$, this trick is no longer possible, so we solve the Thue--Mahler equation with the tools developed in \cite{ThueMahler}. Given a solution $(U, V)$ to the Thue--Mahler equation, it is then elementary to check whether $V = s$ for our fixed value of $s$.

\begin{remark}
    Looking at Lemma \ref{lemma:thuemahleryodd}, it seems that we have a substantial amount of information for $s$. One could ask if we could somehow use this information to simplify the computation if we need to solve a Thue--Mahler equation.
    

    Unfortunately, this is apparently not the case, and it seems that we need to solve the Thue--Mahler equation disregarding our partial knowledge of $s$ and then check if the solution $(U, V)$ satisfies $V=s$.    
    We would like to thank Adela Gherga for a very useful discussion on the subject.
\end{remark}

The previous approach deals with alternatives (i), (ii) and (v) in Corollary \ref{cor:yodd}, so let us consider alternatives (iii) and (iv). In these cases, $p$ divides the class number of $\Q(\sqrt{-c})$, and we can adapt the computational methodology outlined in \textbf{Case II} of Section 6 in \cite{Patel}.

Note that, since we have that $d = q^k$, we do not obtain Thue equations as in \cite{Patel}, but Thue--Mahler equations. Indeed, we get $p-1$ Thue--Mahler equations of degree $p$, of the following shape.
\begin{equation}
    \label{eqn:thuemahlernotcoprime}
    G_2(U, V) = a\cdot q^k,
\end{equation}
where $G_2(U, V) \in \Z[U,V]$ is a homogeneous polynomial of degree $p$, $a \in \Z$ and $k$ is such that $\alpha = 2k$ or $\alpha = 2k+1$. Once this equation is solved, there is an associated expression of the form 
\begin{equation}
    \label{eqn:recoverx}
F_2(U, V) = C_1\cdot a\cdot x,
\end{equation}
where $F_2(U, V) \in \Z[U, V]$ and $U, V$ and $a$ are the same numbers are in \eqref{eqn:thuemahlernotcoprime}. This allows us to recover the value of $x$ in \eqref{eqn:main}.

Note that, as opposed to the case considered in Corollary \ref{cor:yodd}, there is no possibility of avoiding the resolution of Thue--Mahler equations, since we do not have any information on the value of $s$. This is why this case is much more difficult in general.

Note that we need to consider both solutions of \eqref{eqn:thuemahlernotcoprime} where $(U, V)$ are coprime but also those where they are not coprime. In order to solve the first situation, we simply use the Thue--Mahler solver developed in \cite{ThueMahler}, which requires $U$ and $V$ to be coprime.

In the second situation, we need to consider all divisors $d \mid a$ such that $d^p | a$. Then, we solve the equation:
\[G_2(U', V') = \frac{a}{d^p}\cdot q^k,\]
with the Thue--Mahler solver in \cite{ThueMahler} and assuming that $U', V'$ are coprime. We can then recover the original solution $(U, V) = (dU', dV')$, and find the corresponding value of $x$ by \eqref{eqn:recoverx}.

Finally, we note that the Thue--Mahler solver in \cite{ThueMahler} requires the polynomial $G_2(U, V)$ in \eqref{eqn:thuemahlernotcoprime} to be irreducible. If this is not the case, we can factorise it as a product of irreducible coprime factors.
\[G_2(U, V) = a_0 g_1(U, V)\dots g_r(U, V).\]
We may assume, by replacing the value of $a$ in \eqref{eqn:thuemahlernotcoprime} by $a/a_0$, that $a_0 = 1$. Then, it is sufficient to solve the family of Thue equations:
\[g_i(U, V) = d,\]
where $d | a$ and $\gcd(a/d, d ) = 1$. This is a consequence of unique factorisation and simplifies the resolution of the problem significantly, since we only have to deal with Thue equations, as opposed to Thue--Mahler equations. We emphasise that, even if it is not extremely common that $G_2$ is reducible, it does occur sometimes in practice so we have to account for this possibility.

With all these computational techniques, we are finally able to prove Proposition \ref{prop:yodd}.

\begin{proof}(of Proposition \ref{prop:yodd})
    There are $101$ pairs $(C_1, q)$ in the range $1 \le C_1 \le 20$ and $2 \le q \le 25$ with $C_1$ squarefree, $q$ prime and $\gcd(C_1, q) = 1$. Note that, for each pair, we have to consider separately the cases where $\alpha$ is odd and where $\alpha$ is even.
    
    Of these $202$ cases, we have that $p$ divides the class number of $\Q(\sqrt{-c})$ for precisely $10$, and so we land in cases (c) or (d) of Corollary \ref{cor:yodd}. We have to solve $4$ Thue--Mahler equations of degree $5$ for nine of those cases and $6$ Thue--Mahler equations of degree $7$ for one , giving a total of $9 \cdot 4 + 1 \cdot 6 = 42$ Thue--Mahler equations to consider.
    
    For the remaining $192$ pairs, we are able to use the local arguments outlined after the proof of Lemma \ref{lemma:thuemahleryodd} 
    to avoid solving Thue--Mahler equations for $172$ of them. The other $20$ give rise to a total of $54$ Thue--Mahler equations, totalling $96$ Thue--Mahler equations amongst all pairs.
    
    We proceed to solve all of them with the code associated to \cite{ThueMahler}. It is relevant to note that the majority of CPU time is spent solving the $42$ Thue--Mahler equations where $p$ divides the class number of $\Q(\sqrt{-c})$ since, on these cases, we have to account for the possibility that $(U, V)$ are not coprime, which forces the resolution of additional Thue--Mahler equations.
    
    After resolving all the equations, we recover only the four solutions in the statement of the Proposition, and all of them are included in Tables \ref{tab:solutions1} and \ref{tab:solutions2}.
\end{proof}

\section{\texorpdfstring{The case where $y$ is even and $p=5$ or $p=7$: reduction to Thue--Mahler equations}{The case where y is even and p=5 or p=7: reduction to Thue--Mahler equations}}
\label{Sec:ThueMahleryEven}

After the work in Sections \ref{Sec:smallexponents} and \ref{Sec:yodd}, we are left with the case of \eqref{eqn:main} where $y$ is even and $n = p \ge 5$ is a prime number. This case is considerably harder because we may no longer use the theorem of Bilu, Hanrot and Voutier \cite{BHV} on primitive divisors of Lucas-Lehmer sequences, as in Section \ref{Sec:yodd}.

We note that, if $y$ is even, a simple modulo $8$ argument on \eqref{eqn:main} allows us to prove that $C_1q^\alpha \equiv 7 \pmod 8$ and, consequently, either $\alpha$ is odd and $(C_1, q)$ is one of the following pairs:
\begin{align}
\label{eqn:pairsqodd}
    (C_1, q) = & (1, 7), (1, 23), (3, 5), (3, 13), (5, 3), (5, 11), (5, 19), (7, 17), (11, 5),  (11, 13), \\ & (13, 3), (13, 11), (13, 19), (15, 17), (17, 7), (17, 23), (19, 5), (19, 13), \nonumber 
\end{align}
or $\alpha$ is even and $(C_1, q)$ is one of the following pairs:
\begin{align}
\label{eqn:pairsqeven}
    (C_1, q) = & (7, 3), (7, 5), (7, 11), (7, 13), (7, 17), (7, 19), (7, 23) \\ & (15, 7),  (15, 11), (15, 13), (15, 17), (15, 19), (15, 23). \nonumber 
\end{align}
and these are the pairs that we shall consider for the rest of the paper. Even though the pairs $(1, 7)$ and $(1, 23)$ were solved in \cite{BennettSiksek}, our methodology applies to these cases and we will recover the same results.

We shall solve \eqref{eqn:main} with $y$ even and $p \ge 11$ by using an approach combining the modular method, presented in Section \ref{Sec:modularmethod}, with upper bounds on $p$ coming from the theory of linear forms in complex logarithms, that we shall exploit in Section \ref{Sec:LFL}. The success of this methodology has been shown in the following articles, amongst many others: \cite{BennettSiksek}, \cite{BennettSiksek2}, \cite{BMS}.

This leaves only the cases $p = 5$ and $p=7$, which we treat now by explaining how to reduce \eqref{eqn:main} to a Thue--Mahler equation. This will allow us to prove the following Lemma, which completely solves \eqref{eqn:main} in the range $1 \le C_1 \le 20$ and $3 \le q < 25$ if $y$ is even and $p = 5$ or $p = 7$.

\begin{lemma}
    Let $C_1, q$ be integers with $1 \le C_1 \le 20$ and $3 \le q < 25$, with $C_1$ squarefree and $q$ prime. Then, all positive integer solutions $(x, y, \alpha, p)$ to the equation:
    \begin{equation}
        \label{eqn:eqnwithp5}
        C_1x^2 + q^\alpha = y^p, \quad \gcd(C_1x, q, y) = 1, \quad x, y, \alpha > 0, \quad y \text{ even }, \quad p=5, 7. 
    \end{equation}
    are given by the following tuples:
    \begin{align*}
    (C_1, q, x, y, \alpha, p) & =  (1, 7, 5, 2, 1, 5), (1, 7, 181, 8, 1, 5), (1, 7, 11, 2, 1, 7), (1, 23, 3, 2, 1, 5),  \\& (3, 5, 3, 2, 1, 5), (3, 5, 1, 2, 3, 7),  (5, 3, 1, 2, 3, 5), (5, 3, 5, 2, 1, 7),  \\& (7, 5, 1, 2, 2, 5),(7, 11, 1, 2, 2, 7), (13, 11, 3, 2, 1, 7),   (13, 19, 1, 2, 1, 5),  \\& (15, 7, 33, 4, 2, 4), (15, 17, 1, 2, 1, 5), (15, 17, 7, 4, 2, 5), (19, 13, 1, 2, 1, 5).
    \end{align*}
    All of these tuples are included in Tables \ref{tab:solutions1} and \ref{tab:solutions2}.
\end{lemma}

\begin{proof}
We let $c$ and $d$ be defined as in \eqref{eqn:cdalphaodd} if $\alpha = 2k+1$ and as in \eqref{eqn:cdalphaeven} if $\alpha = 2k$. Let $K = \Q(\sqrt{-c})$, $\mathcal{O}_K$ denote its ring of integers, $Cl(K)$ its class group and $h_K$ its class number.
    
We note that, for all pairs in \eqref{eqn:pairsqodd} (if $\alpha$ is odd) and \eqref{eqn:pairsqeven} (if $\alpha$ is even), we have that $-c \equiv 1 \pmod 4$, so that
\[\cO_K = \Z\left[\frac{1+\sqrt{-c}}{2}\right].
\]
Multiplying \eqref{eqn:eqnwithp5} by $C_1$ and dividing by $4$, we obtain the following ideal factorisation in $\mathcal{O}_K$:
    \begin{equation}
        \label{eqn:idealfactorisation}
         \left(\frac{C_1x+q^k\sqrt{-c}}{2} \right)\left(\frac{C_1x-q^k\sqrt{-c}}{2} \right)\mathcal{O}_K = \left(\frac{C_1y^p}{4} \right)\mathcal{O}_K.
    \end{equation}

    Since $\gcd(C_1x, y) = 1$, we see that the only prime ideals dividing both ideals on the left-hand side of \eqref{eqn:idealfactorisation} are those dividing $C_1$. Consequently, we see that
    \begin{equation}
        \label{eqn:factorisation1}
        \left(\frac{C_1x+ q^k\sqrt{-c}}{2} \right) \cO_K = \mathfrak{q} \cdot \mathfrak{p}_2^{p-2} \cdot \mathfrak{A}^p,
    \end{equation}

    where $\mathfrak{q}$ is the product of all prime ideals over $C_1$, $\mathfrak{p}_2$ is one of the two prime ideals over $2$ and $\mathfrak{A}$ is an ideal of norm $y/2$. 
    
    Let $\{\mathfrak{b}_1, \dots, \mathfrak{b}_{h_K} \}$ be a set of representatives for $Cl(K)$ which are integral ideals. Then, it follows that $\mathfrak{A}\mathfrak{b}_i$ is principal for precisely one value of $i = 1, \dots, h_K$. For such an $i$, we let $\beta \in \mathcal{O}_K$ be a generator of $\mathfrak{A}\mathfrak{b}_i$.
    
    If we define $\mathfrak{B} = \mathfrak{q} \cdot \mathfrak{p}_2^{p-2}$, \eqref{eqn:factorisation1} yields that $\mathfrak{B}\mathfrak{b}_i^{-p}$ is a principal fractional ideal, say generated by $\gamma \in K$. Therefore,
    \[\left( \frac{C_1x + q^k\sqrt{-c}}{2} \right)\cO_K = \left( \gamma\beta^p \right)\cO_K,\]
    and, since the units of $\mathcal{O}_K$ are $\pm 1$ for all cases under consideration, we have that, after possibly replacing $\beta$ with $-\beta$, \eqref{eqn:factorisation1} is e\-qui\-va\-lent to
    \begin{equation}
        \label{eqn:thueels}
        C_1x + q^k\sqrt{-c} =  2\gamma\beta^p,
    \end{equation}

    We emphasise that we may compute $\gamma$ explicitly, while we cannot do the same with $\beta$. However, since $\beta \in \cO_K$, we may write 
    \[\beta = U+V\cdot \frac{1+\sqrt{-c}}{2}\]
    for some integers $U, V$. By equating the imaginary parts in \eqref{eqn:thueels} and clearing denominators, we get the following expression:
    \begin{equation}
        \label{eqn:thue}
        a\cdot q^k = F(U, V),
    \end{equation}
    where $a \in \Z$ and $F \in \mathbb{Z}[U, V]$ is a homogeneous polynomial of degree $p$, so that \eqref{eqn:thue} is a Thue--Mahler equation in the variables $U, V$ of degree $p$. If this equation is resolved, we may recover the solution $x$ simply by equating the real values of the expression \eqref{eqn:thueels}, giving rise to an expression of the form:
    \begin{equation}
    \label{eqn:thuesol}
    bx = G(U, V),
    \end{equation}
    for some $b \in \Z$ and certain homogeneous polynomial $G \in \Z[U,V]$ of degree $p$.

    We get one Thue--Mahler equation for each of the pairs in \eqref{eqn:pairsqodd} and \eqref{eqn:pairsqeven} and each exponent $p=5$ or $p=7$. This gives a total of $62$ Thue--Mahler equations to consider. With the help of \texttt{Magma} and the Thue--Mahler solver developed in \cite{ThueMahler}, we solve all of them and recover the solutions via \eqref{eqn:thuesol}.
\end{proof}

\begin{remark}
    In principle, the previous argument would work for arbitrary $p$ and so, in theory, we could reduce \eqref{eqn:main} to the resolution of Thue--Mahler equations.

    However, solving Thue--Mahler equations of degree $p \ge 11$ is an extremely computationally intensive process, and it is practically impossible to carry out if $p \ge 17$. Even so, it is important to emphasise that the same argument applies for $p \ge 11$ and we will use it, in combination with the modular method, in Section \ref{Subsec:modularThue}.
\end{remark}

\section{\texorpdfstring{The case where $y$ is even and $p > 7$: the modular method}{The case where y is even and p > 7: the modular method}}
\label{Sec:modularmethod}

Our main tool to study \eqref{eqn:main} when $y$ is even and $p \ge 11$ will be the modular method for Diophantine equations. An excellent exposition on the modular method and its applications can be found in \cite{Siksek}.

Suppose that $(x, y, \alpha, p)$ is a putative solution to \eqref{eqn:main} with $p \ge 11$ prime and $x, y, \alpha > 0$, with $y$ even. Then, we can associate the following Frey--Hellegouarch curve to it:
\begin{equation}
\label{eqn:frey}
    F_{x, \alpha}: Y^2 + XY = X^3 + \frac{C_1x-1}{4}X^2 + \frac{C_1y^p}{64}X = X^3 + \frac{C_1x-1}{4}X^2 + \frac{C_1^2x^2+C_1q^\alpha}{64}X.
\end{equation}
This Frey--Hellegouarch curve is obtained by applying the recipes of Bennett and Skinner \cite{BenS}, which build upon the work of Wiles, Breuil, Conrad, Diamond and Taylor \cite{modularity, TW, Wiles} on modularity of elliptic curves, on Ribet's level lowering theorem \cite{Ribet}, and on Mazur's theorem \cite{Mazur}. The recipes of Bennett and Skinner are also reproduced in \cite[Section 14.1]{Siksek}. 

Let $f$ be a weight $2$ newform. Then, following \cite{Siksek}, we shall employ the notation $F_{x,\alpha} \sim_p f$ to mean
\[\overline{\rho}_p(F_{x, \alpha}) \cong \overline{\rho}_p(f),\]
where $\overline{\rho}_p(F_{x, \alpha})$ and $\overline{\rho}_p(f)$ are the mod$-p$ Galois representations attached to $F_{x, \alpha}$ and $f$, respectively. Then, by \cite[Theorem 13]{Siksek}, we have that either $0 < \alpha < p$ and $y = 1$ (which would correspond to the case of the curve $F_{x, \alpha}$ having complex multiplication) or $F_{x, \alpha}\sim_p f$ where $f$ is a weight $2$ newform of level
\begin{equation}
\label{eqn:Np}
    N = \begin{cases} 2qC_1^2 & \text{ if } p \nmid \alpha, \\
                    2C_1^2   & \text{ if } p \mid \alpha.
        \end{cases}
\end{equation}
%
It is straightforward to check that there are no solutions to \eqref{eqn:main} with $y = 1$, so from now onwards we may assume that $F_{x, \alpha} \sim_p f$.

For any prime number $\ell$, we define $a_\ell(F) = \ell+1-\#F(\F_\ell)$. Also, we let $f$ have a normalised cuspidal ${q}-$expansion given by:
\begin{equation}
\label{eqn:cuspexpansion}
    f = {q} + \sum_{n=2}^\infty c_n{q}^n,
\end{equation}
where $c_n$ belong to some number field $K_f$, with ring of integers $\mathcal{O}_{K_f}$. Then, a standard consequence of the fact that $F_{x, \alpha} \sim_p f$ (see \cite{KO} and Propositions 5.1 and 5.2 of \cite{Siksek}) is that there is a prime ideal $\mathfrak{p}$ of $\mathcal{O}_{K_f}$ with $\mathfrak{p} \mid p$ and such that for all primes $\ell$, we have that
\begin{equation}
    \label{eqn:congruenceconditions}
    \begin{cases}
        a_\ell(F) \equiv c_\ell \text{ (mod } \mathfrak{p}\text{)} & \text{ if } \ell \neq p \text{ and } \ell \nmid Ny, \\
        c_\ell \equiv \pm(\ell+1) \text{ (mod } \mathfrak{p}\text{)} & \text{ if } \ell \neq 2, p \text{ and } \ell \mid y. \\
    \end{cases}
\end{equation}
Here, the level $N$ is given by \eqref{eqn:Np} and, in both cases, the assumption that $\ell \neq p$ can be removed if the newform is \emph{rational}, that is, if $K_f = \mathbb{Q}$. For each pair $(C_1, q)$ in \eqref{eqn:pairsqodd} (if $\alpha$ is odd) and \eqref{eqn:pairsqeven} (if $\alpha$ is even), we may use \texttt{Magma} to compute the conjugacy classes of rational and irrational newforms of weight $2$ and level $N$ that we need to consider. 

Then, our aim in Sections \ref{Sec:boundexponent} and \ref{Sec:specificexponents} will be to exploit the local information provided by \eqref{eqn:congruenceconditions} in order to prove the following two Propositions. Note that both propositions together cover all pairs in \eqref{eqn:pairsqodd} and \eqref{eqn:pairsqeven}.

\begin{proposition}
    \label{prop:nobigsolutions1}
    Let $C_1, q, \alpha$ and $p$ be positive integers with $C_1$ squarefree, $q$ and $p$ prime numbers and $q \ge 3$. Suppose that one of the following alternatives hold:

    \begin{enumerate}[(i)]
        \item $p \mid \alpha$, or
    \item $\alpha$ is odd and $(C_1, q)$ is one of the following pairs:
        \begin{equation}
        \label{eqn:goodpairsqodd}
        (C_1, q) = (5, 19), (7, 17), (11, 13), (13, 19), (15, 17), (17, 7), (17, 23), (19, 5), (19, 13),
    \end{equation}
    or,
    \item $\alpha$ is even and $(C_1, q)$ is one of the following pairs:
    \begin{equation}
        \label{eqn:goodpairsqeven}
        (C_1, q) = (7, 17), (7, 19), (15, 13), (15, 19), (15, 23).
    \end{equation}
    \end{enumerate}
    Then, there are no solutions $(x, y, \alpha, p)$ to \eqref{eqn:main} with $y$ even and $p \ge 11$.
\end{proposition}

\begin{proposition}
    \label{prop:nobigsolutions2}
    Let $C_1, q$ and $\alpha$ be positive integers with $C_1$ squarefree and $q \ge 3$ prime. Suppose that either $\alpha$ is odd and $(C_1, q)$ is one of the following pairs:
    \begin{equation}
    \label{eqn:badpairsqodd}
    (C_1, q) = (1, 7), (1, 23), (3, 5), (3, 13), (5, 3), \\ (5, 11), (11, 5), (13, 3), (13, 11), 
    \end{equation}
    or $\alpha$ is even and $(C_1, q)$ is one of the following pairs:
    \begin{equation}
    \label{eqn:badpairsqeven}
        (C_1, q) = (7, 3), (7, 5), (7, 11), (7, 13), (7, 23), (15, 7), (15, 11), (15, 17).
    \end{equation}
    Then, define $N_0(C_1, q)$ as follows:    
    \begin{equation}
        \label{eqn:N0}   
        N_0(C_1, q) = 
         \begin{cases}
         7.234157 \cdot 10^7 & \text{ if } (C_1, q) = (1, 7), \\
         1.514725 \cdot 10^8 & \text{ if } (C_1, q) = (1, 23), \\
        3.476178 \cdot 10^7 & \text{ if } (C_1, q) = (3, 5), \\
        1.243438 \cdot 10^8 & \text{ if } (C_1, q) = (3, 13), \\
        3.476178 \cdot 10^7 & \text{ if } (C_1, q) = (5, 3), \\
        8.334595 \cdot 10^7 & \text{ if } (C_1, q) = (5, 11), \\
        7.234157 \cdot 10^7 & \text{ if } (C_1, q) = (7, 3), \\
        7.083124 \cdot 10^7 & \text{ if } (C_1, q) = (7, 5), \\
        7.083124 \cdot 10^7 & \text{ if } (C_1, q) = (7, 11), \\
        7.236925 \cdot 10^7 & \text{ if } (C_1, q) = (7, 13), \\
         7.083124 \cdot 10^7 & \text{ if } (C_1, q) = (7, 23), \\
        8.334595 \cdot 10^7 & \text{ if } (C_1, q) = (11, 5), \\
        1.273969 \cdot 10^8 & \text{ if } (C_1, q) = (13, 3), \\
        3.499196 \cdot 10^8 & \text{ if } (C_1, q) = (13, 11), \\
        3.472013 \cdot 10^7 & \text{ if } (C_1, q) = (15, 7), \\
        3.472013 \cdot 10^7 & \text{ if } (C_1, q) = (15, 11), \\
        3.547538 \cdot 10^7 & \text{ if } (C_1, q) = (15, 17), 
        \end{cases}
    \end{equation}
    Then, the only solutions to \eqref{eqn:main} with $y$ even and $11 \le p \le N_0(C_1, q)$ are given by the following tuples:
    \begin{equation}
        \label{eqn:missingsolutions}
    (C_1, q, x, y, \alpha, p) = (1, 23, 45, 2, 1, 11), (5, 3, 19, 2, 5, 11), (7, 5, 17, 2, 2, 11),
    \end{equation}
    and they are included in Tables \ref{tab:solutions1} and \ref{tab:solutions2}.
    \end{proposition}

\begin{remark}
    In Section \ref{Sec:LFL}, we will prove that, in fact, $p < N_0(C_1, q)$ for each of the cases covered in Proposition \ref{prop:nobigsolutions2}. This, in combination with Propositions \ref{prop:nobigsolutions1} and \ref{prop:nobigsolutions2}, along with our work in previous sections, is enough to finish the proof of Theorem \ref{thm:main}.
\end{remark}
    
\section{\texorpdfstring{Bounding the exponent $p$}{Bounding the exponent p}}
\label{Sec:boundexponent}

In this section, we use the modular method to try to attain a sharp bound for the exponent $p$ in \eqref{eqn:main}. We succeed precisely when $p \mid \alpha$ or if $(C_1, q)$ is one of the pairs in \eqref{eqn:goodpairsqodd} (if $\alpha$ is odd) or in \eqref{eqn:goodpairsqeven} (if $\alpha$ is even). In this situation, we avoid a significantly worse bound coming from linear forms in complex logarithms, so there is a very significant computational improvement. We emphasise that the four techniques that we present are used to obtain a sharp bound for the exponent $p$ for at least one pair $(C_1, q)$.

\subsection{A preliminary modular bound}
 The first method is quite standard and was originally applied by Serre \cite[pp.~203--204]{Serre}. The version that we present here is an adaptation from that of Bennett and Skinner \cite[Proposition 4.3]{BenS}, which is also \cite[Proposition 9.1]{Siksek}. This technique exploits the local information provided by \eqref{eqn:congruenceconditions}, along with the fact that the Frey--Hellegouarch curve \eqref{eqn:frey} has a $\Q-$rational point of order $2$, in order to obtain a bound on $p$. 

\begin{proposition}
    \label{prop:dirbound}
    Suppose that $(x,y,\alpha, p)$ is a solution to \eqref{eqn:main} with $x, y > 0$, $y$ even and $n = p \ge 11$ prime. Let $f$ be a newform of level $N$, where $N$ is given in \eqref{eqn:Np}, with field of coefficients $K_f$. Let $F_{x,\alpha}$ be the Frey--Hellegouarch curve in \eqref{eqn:frey} and suppose that $F_{x,\alpha} \sim_p f$. Then, for any prime number $\ell \nmid N$, we define
    \[
        B'_\ell(f) = \Norm_{K_f/\mathbb{Q}}\left((\ell+1)^2-c_\ell^2 \right)\cdot \prod_{\substack{|a| < 2\sqrt{\ell} \\ 2 \mid a}} \Norm_{K_f/\mathbb{Q}}(a-c_\ell).
    \]
    and
    \[
        B_\ell(f) = \begin{cases}
            B'_\ell(f) & \text{ if } f \text{ is rational.} \\
            \ell B'_\ell(f) & \text{ otherwise.}
        \end{cases}
    \]
    {Then, $p \mid B_\ell(f)$.}
\end{proposition}

\begin{remark}
    \label{rmk:noirrational}
    It is well-known in the literature (see, for example, the considerations after Proposition 9.1 in \cite{Siksek}) that Proposition \ref{prop:dirbound} will succeed in bounding the exponent $p$ if either $f$ is irrational or if $f$ is rational and, via the Modularity Theorem \cite{Wiles}, corresponds to an elliptic curve $E$ which is not isogenous to an elliptic curve with a $\Q-$rational point of order $2$. 
    
    If $f$ is irrational, this is true because $c_\ell \not \in \Q$ for infinitely many values of $\ell$. Therefore, there exists a prime number $\ell$ such that $B_\ell(f) \neq 0$ and we can always bound $p$. If $f$ is rational and the corresponding elliptic curve $E$ is not isogenous to an elliptic curve with a $\Q-$rational point of order $2$, we have by \cite[IV.6]{Serre2} that the set
    \[
        \{\ell \text{ prime} : 2 \mid c_\ell(E) \}
    \]
    is finite. Therefore, we can always find a prime number $\ell$ for which $B_\ell(f) \neq 0$ and once more we can bound the exponent $p$.
    
    Consequently, we shall assume for the remainder of the section that $f$ is a rational newform with corresponding elliptic curve $E$, and we will write $F_{x,\alpha} \sim_p E$ to mean $F_{x,\alpha} \sim_p f$.
\end{remark}

\subsection{An image of inertia argument}
We shall try to disprove $F_{x,\alpha} \sim_p E$ by showing that the corresponding Galois representations have different images for some inertia subgroup of $\Gal\left(\overline{\mathbb{Q}}/\mathbb{Q}\right)$. This approach was originally used in \cite{BenS} and has been used extensively since (see for example \cite{inertia} and \cite{inertia2}). We shall use the following proposition, which is Proposition 4.4 in \cite{BenS}. 

\begin{proposition}(Bennett and Skinner, \cite{BenS})
    \label{prop:imageofinertia}
    Let $\ell \ge 3$ be a prime,  $F_{x,\alpha}$ be the Frey--Hellegouarch curve \eqref{eqn:frey} and $E$ be an elliptic curve such that $F_{x,\alpha} \sim_p E$. Then, the denominator of the $j-$invariant $j(E)$ is not divisible by any prime $\ell \neq p$ dividing $C_1$.
\end{proposition}



\subsection{Studying quadratic twists}
\label{Sec:QTs}
In this section, we will prove that certain quadratic twists of the Frey--Hellegouarch curve $F_{x, \alpha}$ are, again, Frey--Hellegouarch curves. Then, we will use the two Frey--Hellegouarch curves together to try to bound the exponent $p$. The following Proposition goes in this direction and is similar to \cite[Proposition 6.3]{genLN}.

\begin{proposition}
    \label{prop:quadtwist}
    Suppose that $(x,y,\alpha, p)$ is a solution to \eqref{eqn:main} with $y$ even and $p \ge 17$ prime. Also, let $d$ be an integer dividing $C_1$ with $d \equiv 1 \pmod 4$. Then, if we denote by $F_{x, \alpha}^{(d)}$ the quadratic twist of the Frey--Hellegouarch curve \eqref{eqn:frey} by $d$, there exists a newform $f'$ of level $N$, where $N$ is given by \eqref{eqn:Np}, with
    \[F_{x, \alpha}^{(d)} \sim_p f'.\]
\end{proposition}

\begin{proof}
    By combining standard facts about quadratic twists (see \cite{Silverman}) with a careful application of Tate's algorithm (see \cite{Cremona}), we may find that the elliptic curve $F_{x, \alpha}^{(d)}$ has an integral model and conductor given by $N' = 2C_1^2q\Rad_2(y)$ if $p \nmid \alpha$ or $N' = 2C_1^2\Rad_2(y)$ if $p \mid \alpha$. Here, $\Rad_2(y)$ denotes the product of all odd primes dividing $y$.

    Since $p \ge 17$, a result of Mazur \cite{Mazur} implies that the mod-$p$ Galois representation attached to $F^{(d)}_{x,\alpha}$ will be irreducible if $j_{F_{x, \alpha}^{(d)}} \not \in \Z[1/2]$. This is elementary to check and so we conclude that $\overline{\rho}_p(F^{(d)}_{x,\alpha})$ is irreducible.

    Then, Ribet's Level Lowering Theorem \cite{Ribet} yields the existence of a newform $f'$ of level $N$
    such that
    \[\overline{\rho}_p(F^{(d)}_{x,\alpha}) \cong \overline{\rho}_p(f'),\]
    which is precisely the definition of $F^{(d)}_{x, \alpha} \sim_p f'$.
\end{proof}

\begin{remark}
    Let us assume that $F_{x,\alpha} \sim_p E$. Then, Proposition \ref{prop:quadtwist} will give a sharp bound on the exponent $p$ provided that $E^{(d)}$ has a conductor different from $N$. This is due to the fact that 
    \[F_{x, \alpha}^{(d)} \sim_p E^{(d)},\]
    since the corresponding mod$-p$ Galois representations are
    \begin{equation}
        \label{eqn:galoisreps1}
    \overline{\rho}_p(F_{x, \alpha}) \otimes \left(\frac{d}{.}\right) \quad \text{and} \quad 
    \overline{\rho}_p(E) \otimes \left(\frac{d}{.}\right),
    \end{equation}
    where $(d/.)$ denotes the Legendre character. The two Galois representations above are isomorphic because $F_{x, \alpha} \sim_p E$ and, consequently, $\overline{\rho}_p(F_{x, \alpha}) \cong \overline{\rho}_p(E)$. Then, Proposition \ref{prop:quadtwist} yields that
    \[F_{x, \alpha}^{(d)} \sim_p f',\]
    or, equivalently,
    \begin{equation}
        \label{eqn:galoisreps2}
    \overline{\rho}_p(F_{x,\alpha}) \otimes \left(\frac{d}{.}\right) \cong \overline{\rho}_p(f')
    \end{equation}
    Combining \eqref{eqn:galoisreps1} and \eqref{eqn:galoisreps2}, we get that
    \[\overline{\rho}_p(E) \otimes \left(\frac{d}{.}\right) \cong \overline{\rho}_p(f'). \]
    which amounts to the fact that
    \begin{equation}
        \label{eqn:qtsfinalequation}
    E^{(d)} \sim_p f'.
    \end{equation}
    {If the conductor of $E^{(d)}$ is different to the level of $f'$, there will be a prime number $\ell$ with $a_\ell(E^{(d)})-c_\ell(f') \neq 0$. In this instance, we can always bound $p$ by using \eqref{eqn:qtsfinalequation} together with the congruence conditions \eqref{eqn:congruenceconditions}.}     
\end{remark}

\subsection{Using Galois theory}
\label{Subsec:galois}
    In this section, we aim to use Galois theory to refine the technique in Proposition \ref{prop:dirbound}. Note that, in that proposition, the condition $2\mid a$ in the computation of $B_\ell'(f)$ appears merely due to the fact that $F_{x,\alpha}(\mathbb{Q})$ has a point of order $2$ and, consequently, $2 \mid \#F_{x,\alpha}(\mathbb{F}_\ell)$ for all finite fields $\mathbb{F}_\ell$. We will use Galois theory to determine conditions for primes $\ell$ which guarantee that $4 \mid \#F_{x,\alpha}(\mathbb{F}_\ell)$ and, therefore, the condition $2 \mid a$ in Proposition \ref{prop:dirbound} can be strengthened to $4 \mid a$. This is the application of the following proposition, which originally appeared as \cite[Proposition 6.4]{genLN} with a more complicated proof. 
    
\begin{proposition}
    \label{prop:discriminant-trick}
    Let $E$ be an elliptic curve defined over $\mathbb{Q}$ with discriminant $\Delta$, and let $\ell$ be a prime of good reduction for $E$. Furthermore, assume that $E$ has at least one $\mathbb{Q}-$rational point of order $2$. Then, the reduced curve  $\overline{E}(\mathbb{F}_\ell)$ has full 2-torsion, if, and only if its discriminant  ${\Delta}$ \text{is a square mod } $\ell$.
\end{proposition}

\begin{proof}
    Since the curve $E$ has a $\mathbb{Q}-$rational point of order $2$, we may assume that it has a model of the following form:
    \[E: y^2 = f(x) = xg(x),\]
    for some polynomials $f, g \in \mathbb{Z}[x]$ of degree $3$ and $2$ respectively. Since isomorphisms between short Weierstrass models of elliptic curves change the discriminant by a square factor (see Chapter 3 of \cite{Silverman}), it is sufficient to prove the claim for curves in this model. In addition, we can check by direct computation that the discriminant $\Delta$ of $E$ and the discriminant $\Delta_g$ of the polynomial $g(x)$ differ only by a square factor. 

    Let $\alpha \in \overline{\mathbb{F}_\ell}$ be an element with $\alpha^2 \equiv \Delta_g \pmod \ell$ and let $K$ be the splitting field of $\overline{g}(x)$. Since $\overline{g}$ is a quadratic polynomial, we have that
    \[K = \F_\ell(\alpha)\]
    which means that the polynomial $\overline{g}(x)$ will split completely in $\mathbb{F}_\ell$ if, and only if, the discriminant $\Delta$ is a square modulo $\ell$. Since roots of $\overline{g}(x)$ correspond to the remaining $2-$torsion elements of $E(\mathbb{F}_\ell)$, we conclude the proof.
    
    
    
\end{proof}

\begin{remark}
   As stated in the beginning of this section, we will try to compute the quantity $B_\ell(f)$ in Proposition \ref{prop:dirbound} for primes $\ell$ for which the discriminant of the Frey--Hellegouarch curve is a square modulo $\ell$. By \cite[Theorem 16(a)]{Siksek}, this discriminant is given by:
    \[\Delta_{F_{x, \alpha}} = -2^{-12}\cdot C_1^3 \cdot q^\alpha \cdot y^{2n},\]
    and, therefore, it is clearly sufficient to check whether $-C_1q$ or $-C_1$ are squares modulo $\ell$, depending on whether $\alpha$ is odd or even. For these primes, we have that $4 \mid \#F_{x, \alpha}(\mathbb{F}_\ell)$ by Proposition \ref{prop:discriminant-trick} and we can replace the condition $2 \mid a$ in Proposition \ref{prop:dirbound} with the stronger condition $4 \mid a$.
\end{remark}

\subsection{Applying the techniques}

\begin{table}[!ht]
    \centering
    \begin{tabular}{||c c c c c c c c c c ||}
    \hline
$(C_1, q)$ & Level & \multicolumn{2}{c}{No. newforms}  & \multicolumn{2}{c}{\ref{prop:dirbound}} & \ref{prop:imageofinertia} & \ref{prop:quadtwist} & \ref{prop:discriminant-trick} & Remaining \\ 
 \hline\hline
  & & Rat. & Irrat. & Rat. & Irrat. & & & & \\
  \hline 
 $(1, 7)$ & $14$ & $1$ & $0$ & $1$ & $0$ & $1$ & $1$ & $1$ & $1$ \\
 $(1, 23)$ & $46$ & $1$ & $0$ & $1$ & $0$ & $0$ & $0$ & $0$ & $0$ \\
 $(3, 5)$ & $90$ & $3$ & $0$ & $3$ & $0$ & $2$ & $2$ & $2$ & $2$ \\
 $(3, 13)$ & $234$ & $5$ & $0$ & $3$ & $0$ & $2$ & $2$ & $2$ & $2$ \\
 $(5, 3)$ & $150$ & $3$ & $0$ & $3$ & $0$ & $2$ & $2$ & $2$ & $2$ \\
 $(5, 11)$ & $550$ & $13$ & $1$ & $2$ & $0$ & $2$ & $2$ & $2$ & $2$ \\
 $(5, 19)$ & $950$ & $5$ & $9$ & $0$ & $0$ & $0$ & $0$ & $0$ & $0$ \\
 $(7, 17)$ & $1666$ & $14$ & $13$ & $6$ & $0$ & $1$ & $0$ & $0$ & $0$ \\
 $(11, 5)$ & $1210$ & $13$ & $9$ & $2$ & $0$ & $2$ & $2$ & $2$ & $2$ \\
 $(11, 13)$ & $3146$ & $16$ & $21$ & $0$ & $0$ & $0$ &  $0$ & $0$ & $0$ \\
 $(13, 3)$ & $1014$ & $7$ & $8$ & $3$ & $0$ & $2$ & $2$ & $2$ & $2$ \\
 $(13, 11)$ & $3718$ & $20$ & $23$ & $2$ & $0$ & $2$ & $2$ & $2$ & $2$ \\
 $(13, 19)$ & $6422$ & $10$ & $34$ & $1$ & $0$ & $0$ & $0$ & $0$ & $0$ \\ 
 $(15, 17)$ & $7650$ & $68$ & $26$ & $23$ & $0$ & $5$ & $4$ & $0$ & $0$ \\
 $(17, 7)$ & $4046$ & $20$ & $21$ & $8$ & $0$ & $3$ & $2$ & $0$ & $0$ \\
 $(17, 23)$ & $13294$ & $12$ & $42$ & $6$ & $0$ & $3$ & $2$ & $0$ & $0$ \\
 $(19, 5)$ & $3610$ & $9$ & $27$ & $0$ & $0$ & $0$ & $0$ & $0$ & $0$ \\
 $(19, 13)$ & $9386$ & $14$ & $42$ & $1$ & $0$ & $0$ & $0$ & $0$ & $0$ \\
 \hline 
    \end{tabular}
    \caption{Number of conjugacy classes of newforms for each pair $(C_1, q)$ where $p \nmid \alpha$, $\alpha$ is odd and $p$ can be sharply bounded with each technique.}
    \label{tab:newformsbytechniqueqodd}
\end{table}

\begin{table}[!ht]
    \centering
    \begin{tabular}{||c c c c c c c c c c ||}
    \hline
$(C_1, q)$ & Level & \multicolumn{2}{c}{No. newforms}  & \multicolumn{2}{c}{\ref{prop:dirbound}} & \ref{prop:imageofinertia} & \ref{prop:quadtwist} & \ref{prop:discriminant-trick} & Remaining \\ 
 \hline\hline
  & & Rat. & Irrat. & Rat. & Irrat. & & & & \\
  \hline 
  $(7, 3)$ & $294$ & $7$ & $0$ & $3$ & $0$ & $2$ & $2$ & $2$ & $2$ \\
  $(7, 5)$ & $490$ & $11$ & $2$ & $3$ & $0$ & $2$ & $2$ & $2$ & $2$ \\
  $(7, 11)$ & $1078$ & $13$ & $11$ & $5$ & $0$ & $2$ & $2$ & $2$ & $2$ \\
  $(7, 13)$ & $1274$ & $15$ & $8$ & $3$ & $0$ & $2$ & $2$ & $2$ & $2$ \\
 $(7, 17)$ & $1666$ & $14$ & $13$ & $6$ & $0$ & $1$ & $0$ & $0$ & $0$ \\
 $(7, 19)$ & $1862$ & $11$ & $13$ & $0$ & $0$ & $0$ & $0$ & $0$ & $0$ \\
  $(7, 23)$ & $2254$ & $7$ & $21$ & $7$ & $0$ & $3$ & $2$ & $2$ & $2$ \\
 $(15, 7)$ & $3150$ & $44$ & $2$ & $24$ & $0$ & $9$ & $8$ & $4$ & $4$ \\
$(15, 11)$ & $4950$ & $47$ & $15$ & $20$ & $0$ & $8$ & $4$ & $4$ & $4$ \\
$(15, 13)$ & $5850$ & $55$ & $17$ & $25$ & $0$ & $2$ & $0$ & $0$ & $0$ \\
 $(15, 17)$ & $7650$ & $68$ & $26$ & $23$ & $0$ & $5$ & $4$ & $4$ & $4$
 \\
 $(15, 19)$ & $8550$ & $39$ & $37$ & $26$ & $0$ & $2$ & $0$ & $0$ & $0$  \\
  $(15, 23)$ & $10350$ & $49$ & $44$ & $23$ & $0$ & $1$ & $0$ & $0$ & $0$  \\
 \hline  
    \end{tabular}
    \caption{Number of conjugacy classes of newforms for each pair $(C_1, q)$ where $p \nmid \alpha$, $\alpha$ is even and $p$ can be sharply bounded with each technique.}
    \label{tab:newformsbytechniqueqeven}
\end{table}

\begin{table}[!ht]
    \centering
    \begin{tabular}{||c c c c c c c c c c ||}
    \hline
$C_1$ & Level & \multicolumn{2}{c}{No. newforms}  & \multicolumn{2}{c}{\ref{prop:dirbound}} & \ref{prop:imageofinertia} & \ref{prop:quadtwist} & \ref{prop:discriminant-trick} & Remaining \\ 
 \hline\hline
  & & Rat. & Irrat. & Rat. & Irrat. & & & & \\
  \hline 
 $1$ & $2$ & $0$ & $0$ & $0$ & $0$ & $0$ & $0$ & $0$ & $0$ \\
 $3$ & $18$ & $0$ & $0$ & $0$ & $0$ & $0$ & $0$ & $0$ & $0$ \\
 $5$ & $50$ & $2$ & $0$ & $0$ & $0$ & $0$ & $0$ & $0$ & $0$ \\
 $7$ & $98$ & $1$ & $1$ & $1$ & $0$ & $0$ & $0$ & $0$ & $0$ \\
 $11$ & $242$ & $2$ & $4$ & $0$ & $0$ & $0$ & $0$ & $0$ & $0$ \\
 $13$ & $338$ & $6$ & $2$ & $0$ & $0$ & $0$ & $0$ & $0$ & $0$ \\
 $15$ & $450$ & $1$ & $0$ & $1$ & $0$ & $0$ & $0$ & $0$ & $0$ \\
 $17$ & $578$ & $1$ & $8$ & $1$ & $0$ & $0$ & $0$ & $0$ & $0$ \\
 $19$ & $722$ & $6$ & $8$ & $0$ & $0$ & $0$ & $0$ & $0$ & $0$ \\
  \hline 
    \end{tabular}
    \caption{Number of conjugacy classes of newforms for each value of $C_1$ where $p \mid \alpha$ and $p$ can be sharply bounded with each technique.}
    \label{tab:newformsbytechniquepdividesalpha}
\end{table}

Let us apply the four techniques presented in this section in order to achieve a bound on $p$. The following lemma, that we will need to prove Propositions \ref{prop:nobigsolutions1} and \ref{prop:nobigsolutions2}, records our findings.

\begin{lemma}
    \label{lemma:actualbound}
    Let $C_1$ and $q$ be integers with $C_1$ squarefree and $q$ prime. Suppose that $(x, y, \alpha, p)$ is a solution to \eqref{eqn:main} with $x,y > 0$, $y$ even and $p \ge 11$ prime. Then, the following is true:
    \begin{enumerate}[(i)]
        \item If $p \mid \alpha$, then $p \le 19$.
        \item If $p \nmid \alpha$, $\alpha$ is odd and $(C_1, q)$ is one of the pairs in \eqref{eqn:goodpairsqodd}, then $p \le 19$.
        \item If $p \nmid \alpha$, $\alpha$ is even and $(C_1, q)$ is one of the pairs in \eqref{eqn:goodpairsqeven}, then $p \le 47$. 
        \item If $p \nmid \alpha$, $\alpha$ is odd and $(C_1, q)$ is one of the pairs in \eqref{eqn:badpairsqodd}, then either $p \le 19$ or $F_{x, \alpha} \sim_p E_{C_1, q}$, for some $E_{C_1, q}$ whose Cremona reference is recorded in Table \ref{tab:ECsqodd}.
        \item If $p \nmid \alpha$, $\alpha$ is even and $(C_1, q)$ is one of the pairs in \eqref{eqn:badpairsqeven}, then either $p \le 47$ or $F_{x, \alpha} \sim_p E_{C_1, q}$, for some $E_{C_1, q}$ whose Cremona reference is recorded in Table \ref{tab:ECsqeven}.
    \end{enumerate}
\end{lemma}

\begin{proof}
First, let us assume that $p \nmid \alpha$. We then apply the four techniques outlined in this section to all pairs in \eqref{eqn:pairsqodd}, assuming that $\alpha$ is odd, and to all pairs in \eqref{eqn:pairsqeven}, assuming that $\alpha$ is even. Results can be seen in Tables \ref{tab:newformsbytechniqueqodd} and \ref{tab:newformsbytechniqueqeven} respectively, where we record the number of conjugacy classes of newforms for which a sharp bound was not attained after the application of each technique, as well as the number of conjugacy classes of newforms which we were not able to bound. Note that the pairs in \eqref{eqn:goodpairsqodd} and \eqref{eqn:goodpairsqeven} are precisely those for which a sharp bound is attained, while we were unable to bound $p$ using the modular method alone for the pairs in \eqref{eqn:badpairsqodd} and \eqref{eqn:badpairsqeven}.

All four techniques were applied one after the other, and so the latter methods were only used if the former were unsuccessful. If one of the techniques were ommitted, at least one of the pairs in \eqref{eqn:goodpairsqodd} or in \eqref{eqn:goodpairsqeven} would be unboundable, and we would then have to appeal to bounds coming from linear forms in logarithms, which entail more intensive and unnecessary computations.

In all situations, an application of the presented methodology yields that $p \le 19$ for pairs in \eqref{eqn:goodpairsqodd} and that $p \le 47$ for pairs in \eqref{eqn:goodpairsqeven}. For the pairs in \eqref{eqn:badpairsqodd} and \eqref{eqn:badpairsqeven}, we see in Tables \ref{tab:newformsbytechniqueqodd} and \ref{tab:newformsbytechniqueqeven} that there are at most four newforms which are unboundable. For any other newform, the methods outlined successfully yield the same bounds as above. By Remark \ref{rmk:noirrational}, these remaining newforms are necessarily rational. 

Consequently, we may conclude that, for all pairs in \eqref{eqn:badpairsqodd}, either $p \le 19$ or $F_{x, \alpha}\sim_p E_{C_1, q}$, where the Cremona reference of $E_{C_1, q}$ is given in Table \ref{tab:ECsqodd}. Similarly, we have that, for all pairs in \eqref{eqn:badpairsqeven}, either $p \le 47$ or $F_{x, \alpha}\sim_p E_{C_1, q}$, where the Cremona reference of $E_{C_1, q}$ is given in Table \ref{tab:ECsqeven}. We note that, by Proposition \ref{prop:quadtwist}, all the elliptic curves in Tables \ref{tab:ECsqodd} and \ref{tab:ECsqeven} are quadratic twists of each other by some $d \mid C_1$. \\

Finally, we treat the case $p \mid \alpha$. The expression in \eqref{eqn:Np} does not depend on $q$, so we may try to bound $p$ irrespective of the value of $q$ and $\alpha$. We apply the four techniques described in this section and record all results in Table \ref{tab:newformsbytechniquepdividesalpha}. We obtain that $p \le 19$ in all cases.

\end{proof}

\begin{table}[!ht]
    \centering
    \begin{tabular}{|c|c|}
    \hline 
        $(C_1, q)$ & $E_{C_1, q}$ \\
        \hline \hline 
        (1, 7) & 14a1 \\
        \hline
        (1, 23) & 46a1 \\
        \hline 
        (3, 5) & 90a1 or 90b1 \\        
        \hline
        (3, 13) & 234b1 or 234c1 \\
        \hline
        (5, 3) & 150a1 or 150b1 \\
        \hline
        (5, 11) & 550g1 or 550l1 \\
        \hline
        (11, 5) & 1210a1 or 1210h1 \\
        \hline 
        (13, 3) & 1014c1 or 1014g1\\
        \hline 
        (13, 11) & 3718c1 or 3718r1 \\
        \hline 
    \end{tabular}
    
    \caption{Cremona references for the possible elliptic curves $E_{C_1, q}$ for which $F_{x, \alpha} \sim_p E_{C_1, q}$ in the case where $\alpha$ is odd.}
    \label{tab:ECsqodd}
\end{table}

\begin{table}[!ht]
    \centering
    \begin{tabular}{|c|c|}
    \hline 
        $(C_1, q)$ & $E_{C_1, q}$ \\
        \hline \hline
        (7, 3) & 294f1 or 294g1 \\
        \hline
        (7, 5) & 490g1 or 490j1 \\        
        \hline
        (7, 11) & 1078l1 or 1078m1 \\
        \hline
        (7, 13) & 1274j1 or 1274m1 \\
        \hline
        (7, 23) & 2254d1 or 2254e1 \\
        \hline 
        (15, 7) & 3150e1, 3150i1, 3150z1 or 3150bd1 \\
        \hline
        (15, 11) & 4950e1, 4950g1, 4950bb1, 4950bc1 \\
        \hline 
        (15, 17) & 7650h1, 7650i1, 7650bo1 or 7650bp1 \\
        \hline 
    \end{tabular}
    
    \caption{Cremona references for the possible elliptic curves $E_{C_1, q}$ for which $F_{x, \alpha} \sim_p E_{C_1, q}$ in the case where $\alpha$ is even.}
    \label{tab:ECsqeven}
\end{table}

\section{\texorpdfstring{Solving for specific exponents}{Solving for specific exponents}}
\label{Sec:specificexponents}

In this section, we will develop techniques that prove that there are no solutions to \eqref{eqn:main} with $y$ even for a fixed exponent $p$. Then, we will apply these techniques in combination with the results in Lemma \ref{lemma:actualbound} to finish the proof of Propositions \ref{prop:nobigsolutions1} and \ref{prop:nobigsolutions2}.

Note that, in principle, we could simply solve a degree $p$ Thue--Mahler equation, as explained in Section \ref{Sec:ThueMahleryEven}. However, as we have discussed, this gets very computationally intensive and completely impractical for $p \ge 17$.

This is why we will present three techniques that exploit the local information provided by the modular method to prove that there are no solutions to \eqref{eqn:main} for a specific exponent $p$, which are more computationally efficient. Throughout this section, we let $p$ denote a fixed prime exponent.

\subsection{A modification of Kraus's method}
\label{Subsec:modifiedKraus}

The technique that we present here is a combination of the symplectic method, due to Halberstadt and Kraus (Lemme 1.6 of \cite{HalbKraus}), along with a different idea by Kraus \cite{Kraus}, and is inspired by the method called ``Predicting exponents of constants'' in \cite{Siksek}. Before presenting the technique, let us prove an auxiliary Lemma, which gives more detail on the structure of the reduction of the Frey--Hellegouarch curve \eqref{eqn:frey} over $\F_\ell$.

\begin{lemma}
    \label{lemma:twistKraus}
    Let $(x, y, \alpha, p)$ be a solution to \eqref{eqn:main} with $y$ even and $p \ge 11$ prime, and let $\ell$ be a prime number satisfying the following conditions.
    \begin{itemize}
        \item $\ell = 2mp+1$ for some integer $m > 0$.
        \item $\ell \nmid 2qC_1y$.
    \end{itemize}
    Also, let $\beta$ be the unique integer in $\{0,1,\dots, 2p-1\}$ satisfying that $\beta \equiv \alpha \pmod{2p}$. Then, there exists a number $\omega \in \{0,1, \dots, \ell-1\}$ satisfying
    \[(C_1\omega^2 + q^\beta)^{2m} \equiv 1 \pmod \ell,\]
    such that the reduction of the Frey--Hellegouarch curve $F_{x, \alpha}$ over $\F_\ell$ is either isomorphic to the curve
    \[F_{\omega, \beta}{/\F_\ell}: Y^2 + XY = X^3 + \frac{C_1\omega -1}{4}X^2 + \frac{C_1^2\omega ^2+C_1q^\beta}{64}X, \] or a quadratic twist of it by $q \pmod \ell$. 
\end{lemma}

\begin{proof}  
    Let us write $\alpha = 2pu + \beta$ for certain integers $u \ge 0$ and $\beta \in \{0, 1, \dots, 2p-1\}$. Then, and since $\ell \neq q$, we have that $q^{pu} \not \equiv 0 \pmod \ell$, and so we may define $\omega = {x}/{q^{pu}} \pmod \ell$. Then, we see that
    \[y^p = C_1x^2 + q^\alpha \equiv q^{2pu}(C_1\omega^2 + q^\beta) \pmod \ell.\]
    From this, we have that
    \begin{equation}
    \label{eqn:ppower}
    C_1\omega^2 + q^\beta \equiv \left(\frac{y}{q^{2u}} \right)^p \pmod \ell, 
    \end{equation}
    and, since $\ell \nmid y$, it follows that
    \[\frac{y}{q^{2u}} \not \equiv 0 \pmod \ell,
    \]
    and so $(C_1\omega^2 + q^\beta)^{2m} \equiv 1 \pmod \ell$ by Fermat's Little Theorem. Now, let $F_{x, \alpha}^{(1/q^{pu})}$ denote the quadratic twist of the Frey--Hellegouarch curve $F_{x, \alpha}$ by $1/q^{pu}$. Then, a standard formula on quadratic twists (see \cite{Silverman}) yields that
    \[F_{x, \alpha}^{(1/q^{pu})}: Y^2 + XY = X^3 + \frac{C_1x/q^{pu}-1}{4}X^2 + \frac{C_1y^p/q^{2pu}}{64}X.\]
    Now, over $\F_\ell$, the definition of $\omega$, together with \eqref{eqn:ppower}, give that
    \[F_{x, \alpha}^{(1/q^{pu})} \cong Y^2 + XY = X^3 + \frac{C_1\omega-1}{4}X^2 + \frac{C_1^2\omega^2 + C_1q^\beta}{64}X,\]
    which is precisely the expression of $F_{\omega, \beta}$.

    If either $q$ is a square modulo $\ell$ or $u$ is even, then $F_{x, \alpha} \cong F_{x, \alpha}^{(1/q^{pu})} \cong F_{\omega, \beta}$ over $\F_\ell$. Otherwise, $F_{x, \alpha}$ is a quadratic twist of $F_{\omega, \beta}$ by $q \pmod \ell$, as we wanted to show.
\end{proof}

We remark that in Lemma \ref{lemma:twistKraus}, the value of $\beta$ only depends on $p$ and not on $\ell$. This is why, in the following Proposition, we shall use different $\ell$ to show that no $\beta$ can exist and, consequently, that there are no solutions to \eqref{eqn:main}. This will allow us to rule out most exponents $p \le 1000$.

\begin{proposition}
    \label{prop:modifiedKraus}
    Let $(x, y, p, \alpha)$ be a solution to \eqref{eqn:main} with $y$ even and $p \ge 11$. Suppose furthermore that $F_{x, \alpha} \sim_p f$ for some newform $f$ with coefficients $c_n$ as in \eqref{eqn:cuspexpansion} and field of coefficients $K_f$, with ring of integers $\mathcal{O}_{K_f}$. If $K_f = \mathbb{Q}$, we denote by $E$ the minimal model of the elliptic curve associated to $f$ via the Modularity Theorem \cite{Wiles}, and we denote the discriminant of $E$ by $\Delta(E)$.
    
    Let $\ell$ be a prime number satisfying the following conditions:
    \begin{itemize}
        \item $\ell = 2mp+1$ for some integer $m > 0$.
        \item $\ell \nmid 2qC_1$.
        \item Either $N_{K_f/\Q}(c_\ell^2-4) \not \equiv 0 \pmod p$ or $-C_1q^s$ is not a square modulo $\ell$, where $s$ is defined as the unique integer in $\{0,1\}$ satisfying that $s \equiv \alpha \pmod 2$.
    \end{itemize}
    Then, let us define the following sets:
    \begin{equation}
    \label{eqn:defAprime}
    A' = \begin{cases}
    \left\{a \in \{1, \dots, p-1\} \mid \left(\frac{-3a\nu_2(\Delta(E)) \nu_q(\Delta(E))}{p} \right) = 1\right\} & \text{ if } K_f = \mathbb{Q} \text{ and } p \nmid \alpha. \\
    \{0\} & \text{ if } p \mid \alpha. \\
    \{0, \dots, p-1\} & \text{ if } K_f \neq \Q.
    \end{cases}
    \end{equation}
    \begin{equation}
        \label{eqn:defA}
    A = \{\beta \in \{0,\dots, 2p-1\} \mid \beta \equiv a\pmod p \text{ for some } a \in A' \text{ and } \beta \equiv \alpha \pmod 2\},
    \end{equation}
    and the following sets which depend on the prime $\ell$:
    \begin{equation}
    \label{eqn:defXl}
    \mathcal{X}_\ell = \{(\omega, \beta) \in \{0, \dots, \ell-1\} \times A \mid (C_1\omega^2 + q^\beta)^{2m} \equiv 1 \pmod \ell \},
    \end{equation}
    \begin{equation}
        \label{eqn:defYl}
    \mathcal{Y}_\ell = \{(\omega, \beta) \in \mathcal{X}_\ell \mid N_{K/\Q}(a_\ell(F_{\omega, \beta})^2 - c_\ell^2) \equiv 0 \pmod p\}
    \end{equation}
    \begin{equation}
        \label{eqn:defZl}
    \mathcal{Z}_\ell = \{\beta \in \{0, \dots, 2p-1\} \mid (\omega, \beta) \in \mathcal{Y}_\ell \text{ for some } \omega \in \{0, \dots, \ell-1\}\}.
    \end{equation}
  Then, we have that:
  \[\alpha \pmod{2p} \in \bigcap_{\ell} \mathcal{Z}_\ell,\]
  where the intersection is over all prime numbers $\ell$ satisfying the conditions outlined in this Proposition. In particular, if 
  \[\bigcap_{\ell} \mathcal{Z}_\ell = \emptyset,\]
  it follows that there are no solutions $(x, y, p, \alpha)$ to \eqref{eqn:main}.
\end{proposition}

\begin{proof}
    Firstly, we need to show that $\alpha \pmod{2p} \in A$. This follows directly by definition except in the case where $K_f = \mathbb{Q}$ and $p \nmid \alpha$. In this situation, we need to apply the symplectic method, as presented in Section 12 of \cite{Siksek} and based on Lemma 1.6 in \cite{HalbKraus}. 

    In order to apply it, we need the minimal discriminant of the Frey--Hellegouarch curve $F_{x, \alpha}$, which may be computed using Tate's algorithm (see \cite{Cremona}), and is
    \[\Delta({F_{x, \alpha})} = -2^{-12}\cdot C_1^3q^\alpha y^{2p}. \]
    Since $p \nmid \alpha$, the conductor of $F_{x, \alpha}$ is $2C_1^2q\Rad_2(y)$, so the primes $2$ and $q$ have multiplicative reduction and thus the symplectic method yields that 
    \[\frac{\nu_2(\Delta(E))\nu_q(\Delta(E))}{(2p\nu_2(y) - 12)\alpha} \]
    is congruent to a square modulo $p$. From this, it is immediate to conclude that $\alpha \pmod p \in A'$ and, therefore, that $\alpha \pmod{2p} \in A$. \\

    Now, let $\ell$ be a prime with the conditions stated in the Proposition. Suppose first that $\ell \mid y$. Then, by \eqref{eqn:congruenceconditions}, we have that
    \[c_\ell^2 \equiv (\ell+1)^2 \equiv 4 \pmod{ 
    \mathfrak{p}},\]
    for some ideal $\mathfrak{p}$ of $\cO_{K_f}$ over $p$. Here, we used the fact that $\ell = 2mp+1$. This is a contradiction with the third condition on $\ell$.
    
    Consequently, we have that $\ell \nmid y$. Then, we can use Lemma \ref{lemma:twistKraus} to conclude that $F_{x, \alpha}$ is isomorphic, over $\F_\ell$, to certain quadratic twist of $F_{\omega, \beta}$ for some $(\omega, \beta) \in \mathcal{X}_\ell$ satisfying that $\beta \equiv \alpha \pmod{2p}$. 

    Since quadratic twists change $a_\ell$ by a factor of $\pm 1$, it follows that $c_\ell \equiv \pm a_\ell(F_{\omega, \beta}) \pmod{\mathfrak{p}}$ for some prime ideal $\mathfrak{p}$ of $\mathcal{O}_{K_f}$ with $\mathfrak{p} \mid p$, so in particular, we have that
    \[N_{K_f/Q}(a_\ell(F_{\omega, \beta})^2 - c_\ell^2) \equiv 0 \pmod p,\]
    and, consequently, 
    \[\alpha \pmod{2p} \in \mathcal{Z}_\ell, \]
    just like we wanted to show.
\end{proof}

\begin{remark}
    Note that, if $q$ is a square modulo $\ell$, Lemma \ref{lemma:twistKraus} actually yields that $F_{x, \alpha} \cong F_{\omega, \beta}$ over $\F_\ell$, so that the set $\mathcal{Y}_\ell$ in \eqref{eqn:defYl} can  be replaced by 
    \begin{equation}
        \label{eqn:betteryl}
    \mathcal{Y}_\ell = \{(\omega, \beta) \in \mathcal{X}_\ell \mid N_{K_f/\Q}(a_\ell(F_{\omega, \beta})-c_\ell) \equiv 0 \pmod p    \},
    \end{equation}
    which can be stronger in certain cases. We shall exploit this in the upcoming Section.
\end{remark}

\begin{remark}
    Note that we showed in Lemma \ref{lemma:actualbound} that either $p \le 47$ or $f$ is rational and corresponds to an elliptic curve $E$ with a point of order $2$. Since the computational complexity of this method increases with $p$, we shall present a shortcut for the latter case and bigger values of $p$, which, in many cases, will allow us to completely bypass the computation of $a_\ell(F_{\omega, \beta})$. This is inspired by a similar trick in Lemma 14.3 of \cite{BennettSiksek}.

    If $p > 47$, we may assume that $F_{x, \alpha} \sim_p E$ for certain elliptic curve $E$. Since both $E$ and $F_{x, \alpha}$ have a point of order $2$, we have that:
    \[a_\ell(F_{x, \alpha}) \equiv a_\ell(E) \pmod 2.
    \]
    Now, combining this with the fact that $F_{x, \alpha} \sim_p E$ and Lemma \ref{lemma:twistKraus}, we have that \begin{equation}
        \label{eqn:frobtrick}
    a_\ell(F_{\omega, \beta}) \equiv \pm a_\ell(E) \pmod {2p}.
    \end{equation}
    On the other hand, the Hasse bounds (see \cite{Silverman}) yield that:
    \[|a_\ell(F_{\omega, \beta}) \mp a_\ell(E)| < 4\sqrt{\ell},\]
    and, if we furthermore assume that $p^2 > 4\ell$, then \eqref{eqn:frobtrick} implies that $a_\ell(F_{\omega, \beta}) = \pm a_\ell(E)$. In turn, this means that either:
    \[\#F_{\omega, \beta}(\F_\ell) = \#E(\F_\ell),\]
    or
    \[\#F_{\omega, \beta}(\F_\ell) = 2\ell+2-\#E(\F_\ell).\]
    To check whether any of these two equalities hold for a pair $(\omega, \beta)$, we pick a random point $Q \in F_{\omega, \beta}(\F_\ell)$ and check whether either $\#E(\F_\ell)\cdot Q = 0 $ or $(2\ell+2-\#E(\F_\ell))\cdot Q = 0$. If the pair $(\omega, \beta)$ does not pass this test, we do not need to compute $a_\ell(F_{\omega, \beta})$.
\end{remark}

\subsection{Combining the modular method with Thue--Mahler equations}
\label{Subsec:modularThue}
The methodology presented in Proposition \ref{prop:modifiedKraus} is successful for the majority of values of $p \le 1000$. However, it occassionally fails for some values of $p$ which are too large to solve with the Thue--Mahler approach explained in Section \ref{Sec:ThueMahleryEven}. Consequently, we present another method that combines the local information from the modular method with Thue--Mahler equations.

Suppose that $(x, y, p, \alpha)$ is a solution to \eqref{eqn:main} with $y$ even and $p \ge 11$. Then, let us write $\alpha = 2k$ or $\alpha = 2k+1$, depending on whether $\alpha$ is even or odd. Then, \eqref{eqn:thue} and \eqref{eqn:thuesol} yield the following system of equations.
\begin{equation}
        \label{eqn:systemeq}
    \begin{cases}
        aq^k = F(U, V) \\
        bx = G(U, V)
    \end{cases},
\end{equation}
for certain $a, b \in \Z$ and $F, G \in \Z[U, V]$ which are homogeneous polynomials of degree $p$. Let us define $v'(\alpha)$ as the unique integer in $\{0, \dots, p-1\}$ satisfying that
\begin{equation}
    \label{eqn:kmodp}
v'(\alpha) \equiv \begin{cases}
\alpha/2 \pmod p & \text{ if } \alpha \text{ is even.} \\
(\alpha-1)/2 \pmod p & \text{ if } \alpha \text{ is odd.}
\end{cases}
\end{equation}
It follows that $v'(\alpha)$ is the unique integer in $\{0,\dots, p-1\}$ congruent to $k$ modulo $p$, so we may write $k = pu + v'$ for certain $u \ge 0$. Now, if we divide the two equations in \eqref{eqn:systemeq} by $q^{pu}$ and exploit the fact that both $F$ and $G$ are homogeneous polynomials of degree $p$, we obtain
\begin{equation}
\label{eqn:systemeq2}
\begin{cases}
    aq^{v'} = F(U/q^u, V/q^u) \\
    bx/q^{pu} = G(U/q^u, V/q^u).
\end{cases}
\end{equation}
Now, let $\ell$ be a prime number satisfying the following conditions:
\begin{itemize}
     \item $\ell = 2mp+1$ for some integer $m > 0$.
        \item $\ell \nmid 2qC_1$.
\end{itemize}
We note that these are the first two conditions presented in Proposition \ref{prop:modifiedKraus}. Since we excluded the third condition, we now have to take into account the possibility that $\ell \mid y$. Let $A'$ and $A$ be given by \eqref{eqn:defAprime} and \eqref{eqn:defA}. Then, we define the following set.
\[\mathcal{X}'_\ell = \{(\omega, \beta) \in \{0, \dots, \ell-1\} \times A \mid C_1\omega^2 + q^\beta \equiv 0 \pmod \ell \}.
\]
Let $\mathcal{X}_\ell$ be defined as in \eqref{eqn:defXl}, $\mathcal{Y}_\ell$ as in \eqref{eqn:betteryl} if $q$ is a square modulo $\ell$ and as in \eqref{eqn:defYl} otherwise. Finally, we let $\mathcal{Z}_\ell$ be defined by \eqref{eqn:defZl} with $\mathcal{Y}_\ell$ replaced by $\mathcal{Y}_\ell \cup \mathcal{X}'_\ell$. Then, a quick adaptation of our proofs of Lemma \ref{lemma:twistKraus} and Proposition \ref{prop:modifiedKraus} show that:
\[(x/q^{pu}\pmod{\ell}, \quad \alpha\pmod{2p}) \in \mathcal{Y}_\ell \cup \mathcal{X}'_\ell.\]

Consequently, if we reduce \eqref{eqn:systemeq2} modulo $\ell$, we find that there exists a pair $(\omega, \beta) \in \mathcal{Y}_\ell \cup \mathcal{X}'_\ell$ such that the following system of congruence equations has a solution $(\tilde{U}, \tilde{V})$ over $\F_\ell^2$:
\begin{equation}
\label{eqn:congsystem}
S_{\omega, \beta}:
\begin{cases}
    F(\tilde{U}, \tilde{V}) \equiv aq^v \pmod \ell, \\
    G(\tilde{U}, \tilde{V}) \equiv b\omega \pmod \ell.
\end{cases}
\end{equation}
where $v = v'(\beta)$ as defined in \eqref{eqn:kmodp}. Then, we can define the following set:
\[\mathcal{W}_\ell = \{\beta \in \mathcal{Z}_\ell \mid 
\text{ there exists } \omega \text{ with } (\omega, \beta) \in \mathcal{Y}_\ell \cup \mathcal{X}'_\ell\]\[ \text { and } S_{\omega, \beta} \text{ as in \eqref{eqn:congsystem} has solutions over $\F_\ell^2$.}
\}  \]
Our work here shows that, for any solution $(x, y, p, \alpha)$ to \eqref{eqn:main} with $y$ even and $p \ge 11$, it follows that
\[\alpha \pmod{2p} \in \bigcap_\ell \mathcal{W}_\ell, \]
where the intersection is over all primes $\ell$ satisfying the conditions outlined above. Consequently, if 
\[\bigcap_\ell \mathcal{W}_\ell = \emptyset, \]
it follows that \eqref{eqn:main} has no solutions $(x, y, p, \alpha)$. This, as we shall see in Section \ref{Subsec:finishingproof}, allows us to cover many cases of $p$ which were not amenable to Proposition \ref{prop:modifiedKraus}.

While it seems that it would be sufficient to work only with primes $\ell$ satisfying the three conditions in Proposition \ref{prop:modifiedKraus}, we found that allowing $\ell \mid y$ often gave better results. The same is true about extending the definition of $\mathcal{Y}_\ell$ to that in \eqref{eqn:betteryl}, which allows us to cover more cases than the original definition in \eqref{eqn:defYl}.

\subsection{A method for bigger exponents}
\label{Subsec:biggerexponents}
The combination of the two previous techniques is very successful in ruling out the existence of solutions for relatively small exponents ($p < 1000$). Unfortunately, it is very computationally expensive for bigger exponents and is completely unfeasible\footnote{To illustrate this, we note that an application of the previous two techniques to prove the non-existence of solutions to \eqref{eqn:main} where $C_1 = 3, q = 5$, $p = 1,000,033$ and $y$ is even takes over 107 minutes on a 3 GHz Intel Xeon E5-2623 processor.} for $p > 10^5$, which justifies the need for a new methodology, which we present here. This section is a modification of Lemmas 15.5 and 15.6 in \cite{BennettSiksek}, and is essentially a refinement of Proposition \ref{prop:modifiedKraus} over number fields.

For this, we recover some notation from Section \ref{Sec:ThueMahleryEven}. Suppose that $(x, y, \alpha, p)$ is a solution to \eqref{eqn:main} with $y$ even. Let $c$ be given as in \eqref{eqn:cdalphaodd}, if $\alpha$ is odd,  or by \eqref{eqn:cdalphaeven}, if $\alpha$ is even, and let $K = \Q(\sqrt{-c})$, with corresponding ring of integers $\mathcal{O}_K$, class group $Cl(K)$ and class number $h_K$. We recall that, by \eqref{eqn:factorisation1}, we have that
\begin{equation}
    \label{eqn:factorisation2}
        \left(\frac{C_1x+ q^k\sqrt{-c}}{2} \right) \cO_K = \mathfrak{q} \cdot \mathfrak{p}_2^{p-2} \cdot \mathfrak{A}^p,
\end{equation}
where $\mathfrak{q}$ is the product of all prime ideals over $C_1$, $\mathfrak{p}_2$ is one of the two prime ideals over $2$ and $\mathfrak{A}$ is some ideal of norm $y/2$.

At this point, it is important to remark that, for all values of $C_1$ and $q$ under consideration, the group $Cl(K)$ is cyclic and generated by the class of $\mathfrak{p}_2$. Similarly, because $p > 1000$, we certainly have that $p \nmid h_K$ in all cases. These are key ingredients in the proof of the following Lemma, which transforms \eqref{eqn:factorisation2} into an expression about elements in $K$.

\begin{lemma}
    \label{lemma:elementfactorisation}
    Keeping notation as above, assume that $p > 1000$. Let $j$ be the unique integer in $\{0, 1 , \dots, h_K - 1 \}$ satisfying
    \[\mathfrak{q}\mathfrak{p}_2^j = \omega\cO_K
    \]
    for certain $\omega \in \cO_K$. Let $i$ be the unique integer in $\{0, 1, \dots, h_K-1\}$ satisfying
    \[pi \equiv -2-j \pmod{h_K}.
    \]
    Define $\delta \in \cO_K$ to be such that $\mathfrak{p}_2^{h_K} = \delta\cO_K$ and $n^*$ by the expression
    \[n^* = \frac{-2-j-pi}{h_K}.\]
    Then,
    \begin{equation}
        \label{eqn:elementfactorisation}
    \frac{C_1x+q^k\sqrt{-c}}{2} = \omega \gamma^p \delta^{n^*},
    \end{equation}
    for some element $\gamma \in \cO_K$ of norm $2^iy$.
\end{lemma}

\begin{proof}
    Note that we may rewrite \eqref{eqn:factorisation2} as
    \begin{equation}
        \label{eqn:factorisation3}
    \left(\frac{C_1x+q^k\sqrt{-c}}{2}\right)\cO_K = \mathfrak{q}\mathfrak{p}_2^{-2}\mathfrak{B}^p,  
    \end{equation}
    where $\mathfrak{B}$ is some ideal of norm $y$. Since, as mentioned prior to the statement of this Lemma, $\mathfrak{p}_2$ is a generator for $Cl(K)$, it follows that there are unique integers $i, j \in \{0, \dots, h_K-1\}$ satisfying
    \begin{equation}
        \label{eqn:idealfactorisation2}
    \mathfrak{B}\mathfrak{p}_2^i = \gamma \cO_K \quad \text{and} \quad \mathfrak{q}\mathfrak{p}_2^j = \omega \cO_K.
    \end{equation}
    for some $\gamma, \omega \in \cO_K$. Note that since $\mathfrak{B}$ has norm $y$ and $\mathfrak{p}_2$ has norm $2$, we have that
    $N(\gamma) = 2^iy$. Then, it follows from \eqref{eqn:factorisation3} and \eqref{eqn:idealfactorisation2} that $\mathfrak{p}_2^{-2-j-pi}$ is principal. Now, because $\mathfrak{p}_2$ is a generator of $Cl(K)$, we have that $h_K \mid -2-j-pi$, and, consequently, $j \equiv -2 -pi \pmod{h_K}$, and so $i$ is the unique integer in $\{0, 1, \dots, h_K-1\}$ satisfying $pi \equiv -2-j \pmod{h_K}$. We may then rewrite \eqref{eqn:factorisation3} as follows.
    \[\frac{C_1x + q^k\sqrt{-c}}{2}\cO_K = (\mathfrak{q}\mathfrak{p}_2^j)(\mathfrak{p}_2^i\mathfrak{B})^p\mathfrak{p}_2^{-2-j-pi},
    \]
    and then \eqref{eqn:elementfactorisation} follows by definition of $\delta, \omega, \gamma$ and $n^*$. 
\end{proof}
With the previous Lemma, we are finally able to present a method to prove that there are no solutions to \eqref{eqn:main} for a given prime $p$, which is furthermore computationally feasible.

\begin{proposition}
    \label{prop:kraushigherexponents}
    Let $(x, y, p, \alpha)$ be a solution to \eqref{eqn:main} with $y$ even and $p > 1000$, and let $\omega, \delta$ and $n^*$ be given as in Lemma \ref{lemma:elementfactorisation}. Denote by $E = E_{C_1, q}$ the curve given in Table \ref{tab:ECsqodd} if $\alpha$ is odd or in Table \ref{tab:ECsqeven} if $\alpha$ is even. Consider a prime number $\ell \nmid 2C_1q$ satisfying the following conditions:

    \begin{enumerate}[(I)]
        \item $\ell = 2mp + 1$ for some integer $m > 0$.
        \item $\left(\frac{-c}{\ell}\right) = 1, $ so that the prime $\ell$ is split in $K = \Q(\sqrt{-c})$.
        \item $a_{\ell}(E)^2 \not \equiv 4 \pmod{p}.$
    \end{enumerate}
    Then, let $\mathfrak{L}$ be one of the two prime ideals over $\ell$ in $\cO_K$, and let $\F_{\mathfrak{L}} = \cO_K/\mathfrak{L} \cong \F_\ell$ denote the residue field. Let $g$ be a generator of $\F_{\mathfrak{L}}^*$, define $h = g^p$ and let ${\beta}$ and $\theta$ be given by
    \[
    {\beta} \equiv \frac{\overline{\delta}}{\delta} \pmod{\mathfrak{L}} \quad \text{and} \quad \theta \equiv \frac{\overline{\omega}}{\omega} \pmod{\mathfrak{L}}.
    \]
    We then define the following set:
    \[
    \mathcal{X}_{\ell, p} = \{\theta \cdot {\beta}^{n^*}\cdot h^j \mid j = 0, 1, \dots, 2m-1 \} \setminus \{\overline{1}\}.
    \]
    For $\tau \in \mathcal{X}_{\ell, p}$, define the following elliptic curve defined over $\F_{\mathfrak{L}} \cong \F_\ell$:
    \[E_{\tau}: Y^2 = X(X+1)(X+\tau).
    \]
    Finally, let $\mathcal{Z}_{\ell, p}$ be the set defined by
    \[
    \mathcal{Z}_{\ell, p} = \{\tau \in \mathcal{X}_{\ell, p} \mid a_{\mathfrak{L}}(E_{\tau})^2 \equiv a_{\ell}(E)^2 \pmod p \}.
    \]
    Then, $\mathcal{Z}_{\ell, p} \neq \emptyset$.
\end{proposition}

\begin{proof}
    Remember that $F_{x, \alpha} \sim_p E$, where $F_{x, \alpha}$ is the Frey--Hellegouarch curve given in \eqref{eqn:frey}. First, we will show that $\ell \nmid y$. Suppose otherwise. Then, the congruence conditions \eqref{eqn:congruenceconditions} yield that
    \[a_\ell(E) \equiv \pm(\ell+1) \pmod p.
    \]
    If we combine this with condition (I) on $\ell$, we get that
    \[a_\ell(E) \equiv \pm 2 \pmod p,\]
    which is a contradiction with condition (III). Consequently, $\ell \nmid y$. Thus, the congruence conditions \eqref{eqn:congruenceconditions} imply that $a_\ell(F_{x, \alpha}) \equiv a_\ell(E) \pmod p$. Note that, apart from the model in \ref{eqn:frey}, $F_{x, \alpha}$ also has the following model:
    \[F_{x, \alpha}: Y^2 = X\left(X^2 + \frac{C_1x}{4}X + \frac{C_1y^p}{64}\right),
    \]
    which can be rewritten over $K$ as
    \begin{equation}
        \label{eqn:freysplit}
    F_{x, \alpha}: Y^2 = X\left(X + \frac{C_1x + q^k\sqrt{-c}}{8}\right)\left(X + \frac{C_1x - q^k\sqrt{-c}}{8}\right).
    \end{equation}
    By Lemma \ref{lemma:elementfactorisation}, we have that:
    \[\frac{C_1x - q^k\sqrt{-c}}{C_1x + q^k\sqrt{-c}} = \frac{\overline{\omega} \cdot  \overline{\delta}^{n^*}\cdot \overline{\gamma}^p}{\omega\cdot {\delta}^{n^*} \cdot{\gamma}^p} = \left(\frac{\overline{\omega}}{\omega}\right)\left(\frac{\overline{\delta}}{\delta}\right)^{n^*}\left(\frac{\overline{\gamma}}{\gamma}\right)^p \equiv \theta \cdot {\beta}^{n^*} \cdot h^j  \pmod{\mathfrak{L}}
    \]
    for some $j = 0, 1, \dots, 2m-1$. Note that $\theta \cdot \beta^{n^*}\cdot h^j \not \equiv 1 \pmod{ \mathfrak{L}}$ since, if this were the case, we would have that $\mathfrak{L} \mid (2q^k\sqrt{-c})\cO_K$, which is a contradiction with the fact that $\ell \nmid 2C_1q$. Therefore, it follows that
    \[\frac{C_1x - q^k\sqrt{-c}}{C_1x + q^k\sqrt{-c}} \pmod{\mathfrak{L}} \in \mathcal{X}_{\ell, p}.
    \]
    Combining this with \eqref{eqn:freysplit}, we may see that, over $\mathbb{F}_{\mathfrak{L}}$, $F_{x, \alpha}$ is a quadratic twist of $E_\tau$ for certain $\tau \in \mathcal{X}_{\ell, p}$. Thus,
    \[a_{\mathfrak{L}}(E_\tau)^2 \equiv a_\ell(E)^2 \pmod p
    \]
    for some $\tau \in \mathcal{X}_{\ell, p}$ and it readily follows that the set $\mathcal{Z}_{\ell, p}$ is non-empty.
    
\end{proof}

\begin{remark}
    In order to show that \eqref{eqn:main} has no solution for a fixed value of $p$, it is therefore sufficient to prove that $\mathcal{Z}_{\ell, p} = \emptyset$ for some prime $\ell$ satisfying conditions (I), (II) and (III) in \ref{prop:kraushigherexponents}.

    Note that, in principle, the proof of Proposition \ref{prop:kraushigherexponents} does not require $p > 1000$ to work. However, we have found that this technique is often unsuccessful in practice for smaller values of $p$, forcing us to use the sieving methods developed in Sections \ref{Subsec:modifiedKraus} and \ref{Subsec:modularThue}.
\end{remark}

\subsection{Finishing the proof of Propositions \ref{prop:nobigsolutions1} and \ref{prop:nobigsolutions2}}
\label{Subsec:finishingproof}
With the techniques that we have developed in this section, we can finish the proof of Propositions \ref{prop:nobigsolutions1} and \ref{prop:nobigsolutions2}.

\begin{proof}(of Proposition \ref{prop:nobigsolutions1})
By Lemma \ref{lemma:actualbound}, it is sufficient to consider the range $11 \le p \le 19$ for cases (i) and (ii) in the Proposition and the range $11 \le p \le 47$ for case (iii). We apply the techniques developed in Sections \ref{Subsec:modifiedKraus} and \ref{Subsec:modularThue} for each value of $p$ in the corresponding ranges. This successfully proves that there are no solutions in all cases but two, corresponding to the tuples
\[(C_1, q, p) = (11, 13, 11), (19, 5, 11),\]
$\alpha$ odd and $p \nmid \alpha$. Consequently, we need to solve two Thue--Mahler equations of degree $11$, as in Section \ref{Sec:ThueMahleryEven}. Once more, we employ the Thue--Mahler solver developed in \cite{ThueMahler} and recover no solutions to \eqref{eqn:main}, finishing the proof of the Proposition.
\end{proof}

\begin{proof}(of Proposition \ref{prop:nobigsolutions2})
By Lemma \ref{lemma:actualbound}, for all pairs in \eqref{eqn:badpairsqodd}, corresponding to $\alpha$ odd, we have that either $p \le 19$ or $F_{x, \alpha} \sim_p E_{C_1, q}$ for the curves $E_{C_1, q}$ given in Table \ref{tab:ECsqodd}. We deal with the first case identically as in the proof of Proposition \ref{prop:nobigsolutions1}. For the second case, we employ the methodology presented in Sections \ref{Subsec:modifiedKraus} and \ref{Subsec:modularThue} for $11 \le p < 1000$ and Proposition \ref{prop:kraushigherexponents} for $1000 < p < N_0(C_1, q)$. We performed the computation on a 3 GHz Intel Xeon E5-2623 and the necessary times are recorded on Table \ref{tab:comptimesalphaodd}. This method is successful in all but two cases, corresponding to the tuples
\[(C_1, q, p) = (1, 23, 11), (5, 3, 11).\]
The resolution of the corresponding Thue--Mahler equations gives rise to the following solutions:
\[(C_1, q, x, y, \alpha, p) = (1, 23, 1, 5, 1, 11), (5, 3, 19, 2, 5, 11).\]
Similarly, if $\alpha$ is even, Lemma \ref{lemma:actualbound} yields that, for all pairs in \eqref{eqn:badpairsqeven}, either $p \le 47$ or $F_{x, \alpha} \sim_p E_{C_1, q}$ for the curves $E_{C_1, q}$ in Table \ref{tab:ECsqeven}. We follow the same computational approach, where the CPU times on the same processor as above are recorded in Table \ref{tab:comptimesalphaeven}. We succeed in all cases except two, corresponding to the pairs
\[(C_1, q, p) = (7, 5, 11), (7, 23, 11),\]
and whose corresponding Thue--Mahler equations gives rise only to the solution
\[(C_1, q, x, y, \alpha, p) = (7, 5, 17, 2, 2, 11).\]
Since the three solutions that we recovered are precisely those in \eqref{eqn:missingsolutions}, we conclude the proof of the Proposition.
\end{proof}

\begin{table}[!ht]
    \centering
    \begin{tabular}{||c|c||}
            \hline 
             $(C_1, q)$ & CPU time needed to get to $N_0(C_1, q)$ \\
             \hline\hline 
             $(1, 7)$ &  22.96 hours \\
             \hline
             $(1, 23)$ & 54.26 hours \\
             \hline 
             $(3, 5)$ &  10.47 hours\\
             \hline 
             $(3, 13)$ & 43.55  hours\\
             \hline
             $(5, 3)$ & 13.51 hours \\
             \hline
             $(5, 11)$ & 32.89 hours \\
             \hline
             $(11, 5)$ & 33.65 hours \\
             \hline 
             $(13, 3)$ & 51.60 hours \\
             \hline
             $(13, 11)$ & 168.55 hours \\
             \hline 
     \end{tabular}
    \caption{Required time to prove that there are no solutions for  $11 \le p < N(C_1, q)$ for $\alpha$ odd.}
    \label{tab:comptimesalphaodd}
\end{table}

\begin{table}[!ht]
    \centering
    \begin{tabular}{||c|c||}
            \hline 
             $(C_1, q)$ & CPU time needed to get to $N_0(C_1, q)$ \\
             \hline\hline 
             $(7, 3)$ & 27.37  hours \\
             \hline
             $(7, 5)$ & 27.13  hours\\
             \hline 
             $(7, 11)$ & 27.35  hours\\
             \hline
             $(7, 13)$ & 26.95 hours \\
             \hline
             $(7, 23)$ & 22.36 hours \\
             \hline 
             $(15, 7)$ & 14.02 hours \\
             \hline
             $(15, 11)$ & 13.87 hours \\
             \hline 
             $(15, 17)$ & 13.87 hours \\
             \hline
     \end{tabular}
    \caption{Required time to prove that there are no solutions for  $11 \le p < N(C_1, q)$ for $\alpha$ even.}
    \label{tab:comptimesalphaeven}
\end{table}

\section{Linear forms in logarithms}
\label{Sec:LFL}
After proving Propositions \ref{prop:nobigsolutions1} and \ref{prop:nobigsolutions2}, our aim in this section is to prove that, for pairs in \eqref{eqn:badpairsqodd} and \eqref{eqn:badpairsqeven}, we have that $p < N_0(C_1, q)$, therefore completing the proof of Theorem \ref{thm:main}. This is the content of Proposition \ref{prop:lflub}, which will be the main result in this section.

Throughout this section, and by Proposition \ref{prop:nobigsolutions2}, we shall assume that $p > 3\cdot 10^7$ in all cases. 

\begin{proposition}
    \label{prop:lflub}
     Let $(C_1, q)$ be one of the pairs in \eqref{eqn:badpairsqodd}, if $\alpha$ is odd, or in \eqref{eqn:badpairsqeven}, if $\alpha$ is even and let $(x, y, p, \alpha)$ be a solution to \eqref{eqn:main} with $y$ even. Then, $p < N_0(C_1, q)$, where $N_0(C_1, q)$ is given in \eqref{eqn:N0}.
\end{proposition}

In order to prove this, we will need to use the new estimates on lower bounds on linear forms in three complex logarithms available in \cite{Voutier}, as well as estimates on linear forms in $q$-adic logarithms, following \cite{BennettSiksek}. Before that, we shall build upon our work in Section \ref{Subsec:biggerexponents} to prove the following lemma. 

\begin{lemma}
    \label{lemma:factlfl}
    Let $(x, y, \alpha, p)$ be a solution to \eqref{eqn:main} with $y$ even and $p > 3\cdot 10^7$. Let $c, K, \cO_K$, $h_K$ and $\mathfrak{p}_2$ be given as in Section \ref{Subsec:biggerexponents}, and define $s$ to be the smallest positive integer such that the ideal $\mathfrak{p}_2^{2s}$ is principal, say generated by $\delta \in \cO_K$. Then,
    \[\left(\frac{C_1x-q^k\sqrt{-c}}{C_1x+q^k\sqrt{-c}}\right)^s = \beta\cdot \gamma^p,\]
    for some $\gamma \in K$ and $\beta \in K$ given by
    \begin{equation}
        \label{eqn:defbeta}
    \beta = \frac{\delta}{\overline{\delta}}.
    \end{equation}
\end{lemma}

\begin{proof}
    By \eqref{eqn:factorisation3}, we have that
    \[\left(\frac{C_1x+q^k\sqrt{-c}}{2}\right)\cO_K = \mathfrak{q}\mathfrak{p}_2^{-2}\mathfrak{B}^p,
    \]
    where $\mathfrak{q}$ is the product of all prime ideals over $C_1$ and $\mathfrak{B}$ is some ideal of norm $y$. Then, we see that
    \[\left(\frac{C_1x-q^k\sqrt{-c}}{C_1x+q^k\sqrt{-c}}\right)\cO_K = \left(\frac{\mathfrak{p}_2}{\overline{\mathfrak{p}_2}}\right)^2 \left(\frac{\overline{\mathfrak{B}}}{\mathfrak{B}}\right)^p,
    \]
    so that, by definition of $\beta$, we have that
        \begin{equation}
            \label{eqn:sfactorisation}
        \left(\frac{C_1x-q^k\sqrt{-c}}{C_1x+q^k\sqrt{-c}}\right)^s\cO_K = \beta\cO_K\cdot \left(\frac{\overline{\mathfrak{B}}}{\mathfrak{B}}\right)^{sp}.
    \end{equation}
    It readily follows that the fractional ideal \[\left(\frac{\overline{\mathfrak{B}}}{\mathfrak{B}}\right)^{sp}\]
    is principal. However, since $p > 3\cdot10^7$, $p$ does not divide the class number $h_K$ for any of the cases under consideration. Consequently, the ideal
    \[\left(\frac{\overline{\mathfrak{B}}}{\mathfrak{B}}\right)^{s}\]
    is principal, say generated by $\gamma \in K$. This, combined with \eqref{eqn:sfactorisation}, shows that:
    \[\left(\frac{C_1x-q^k\sqrt{-c}}{C_1x+q^k\sqrt{-c}}\right)^s = \pm \beta \gamma^p,\]
    since the units of $\cO_K$ are $\pm 1$. Replacing $\gamma$ by $-\gamma$ if necessary, this finishes the proof.
\end{proof}

\begin{remark}
    As we mentioned immediately before Lemma \ref{lemma:elementfactorisation} in Section \ref{Subsec:biggerexponents}, $\mathfrak{p}_2$ is a generator for $Cl(K)$, so either $s=h_K$ (if $h_K$ is odd) or $s=h_K/2$.
\end{remark}

In order to obtain a bound for the exponent $p$, we need a lower bound for $y$, which is given in the following Lemma. Its proof is identical to \cite[Lemma 6.1]{BennettSiksek}, so we omit it.

\begin{lemma}
\label{lemma:boundy}
Let $(x, y, \alpha, p)$ be a solution to \eqref{eqn:main}, with $y$ even and where $y$ is not a power of $2$. Then, we have that
\[ y > 4p-4\sqrt{2p}+2.\]
\end{lemma}

We note that we can safely assume that $y$ is not a power of $2$. This is because, if $y=2^m$, then the Frey--Hellegouarch curve $F_{x, \alpha}$ given in \eqref{eqn:frey} has conductor equal to $2C_1^2q$ and minimal discriminant given by $\Delta = -2^{2pm-12}\cdot C_1^3q$ and so we can determine $m$ simply by inspecting Cremona's tables (\cite{Cremona}) for the corresponding conductor. All these solutions were previously obtained in Sections \ref{Sec:ThueMahleryEven} and \ref{Sec:specificexponents}.

\subsection{\texorpdfstring{Linear forms in $q$-adic logarithms}{Linear forms in q-adic logarithms}}
In order to obtain an upper bound for $p$ using the tools in \cite{Voutier}, we previously need an upper bound for $\alpha$. For this, we need to appeal to upper bounds on linear forms in $q$-adic logarithms, as in the following lemma.

\begin{lemma}
    \label{lemma:qadiclogs}
    Let $(x, y, p, \alpha)$ be a solution to \eqref{eqn:main} with $y$ even, $p > 3\cdot 10^7$ and $\alpha = 2k$ or $\alpha = 2k+1$. Also, let $s$ and $K = \Q(\sqrt{-c})$ be defined as in Lemma \ref{lemma:factlfl}. In the notation of Theorem 10 in \cite{BennettSiksek}, let $f$ be the residual degree of the extension $\Q_q(-c)/\Q_q$, and let $D = [\Q_q(-c):\Q_q]/f$. Finally, let $A_2$ be given by:
    \[\log(A_2) = \max\left\{s\log(2), \frac{\log(q)}{D}\right\}.
    \]
    Then, we have that
    \begin{multline*}    
    k \le \frac{48qs}{\log^4(q)}\frac{q^f-1}{q-1}D^2    \log(A_2) \max\{ \log(p) + \log(\log(q)) \\-\log(D\log(A_2))+0.401, 5\log(q)\}^2\log(y).
    \end{multline*}
\end{lemma}

\begin{proof}
    By Lemma \ref{lemma:factlfl}, we have that
    \[\left(\frac{C_1x-q^k\sqrt{-c}}{C_1x+q^k\sqrt{-c}} \right)^s-1 = \beta \gamma^p - 1.
    \]
    Now, we may rewrite the left-hand side of this equality as
    \begin{equation}
        \label{eqn:refactorisation}
    \left(\frac{C_1x-q^k\sqrt{-c}}{C_1x+q^k\sqrt{-c}} \right)^s-1 = \left(\frac{-2q^k\sqrt{-c}}{C_1x+q^k\sqrt{-c}}\right)\sum_{i=0}^{s-1} \left(\frac{C_1x-q^k\sqrt{-c}}{C_1x+q^k\sqrt{-c}}\right)^i.
    \end{equation}
    We note that, by \eqref{eqn:defbeta}, $\beta \in K$ satisfies $\beta \overline{\beta} = 1$ and is supported only on primes above $2$, while the proof of Lemma \ref{lemma:factlfl} shows that $\gamma$ is supported only on primes above $y$. Since $y$ is even and $\gcd(q, y) = 1$, it follows that $\nu_{\mathfrak{q}}(\beta) = \nu_{\mathfrak{q}}(\gamma) = 0$. Then, multiplying \eqref{eqn:refactorisation} by $\overline{\beta}$ and using the fact that $\gcd(C_1x, q) = 1$ yields that 
    \[\nu_{\mathfrak{q}}(\gamma^p-\overline{\beta}) \ge k,\]
    and it is therefore sufficient to obtain an upper bound for $\nu_{\mathfrak{q}}(\gamma^p-\overline{\beta})$. For this purpose, we shall use Theorem 10 in \cite{BennettSiksek}, which is due to Bugeaud and Laurent and was originally presented in \cite{qadiclogs}. In the notation of this result, we let
    \[\alpha_1 = \gamma, \quad \alpha_2 = \overline{\beta}, \quad b_1 = p, \quad b_2 = 1.
    \]
    By Lemma 13.2 of \cite{BMS}, we can compute the absolute logarithmic heights of $\gamma$ and $\overline{\beta}$ and see that
    \[h(\gamma) = \frac{s}{2}\log(y) \quad \text{and} \quad h(\overline{\beta}) = s\log(2),\]
    and, consequently, in the notation of Theorem 10 in \cite{BMS}, we may select the a\-ppro\-piate pa\-ra\-me\-ters as follows.
    \[\log(A_1) = \max\left\{h(\gamma), \frac{\log(q)}{D}\right\} = \frac{s}{2}\log(y),
    \]
    by the lower bounds on $y$ given in Lemma \ref{lemma:boundy}. Similarly,
    \[\log(A_2) = \max\left\{h(\overline{\beta}), \frac{\log(q)}{D}\right\} = \max\left\{s\log(2), \frac{\log(q)}{D}\right\},
    \]
    which is precisely the expression for $\log(A_2)$ given in the statement of this Lemma. Then, an application of Theorem 10 in \cite{BennettSiksek} yields that
    \begin{multline}
        \label{eqn:ubqadiclogs}
    \nu_{\mathfrak{q}}(\gamma^p-\overline{\beta}) \le \frac{12qs}{\log^4(q)}\frac{q^f-1}{q-1}D^4 \log(A_2) \\ \max\left\{ \log(b')  + \log(\log(q)) + 0.4, \frac{10}{D}\log(q)\right\}^2\log(y). 
    \end{multline}
    where $b'$ is given by
    \[b' = \frac{p}{D\log(A_2)} + \frac{2}{Ds\log(y)}.\]
    Note that the lower bound on $y$ given by Lemma \ref{lemma:boundy}, combined with the fact that $p > 3 \cdot 10^7$, give that:
    \[b' \le \frac{1.001}{D\log(A_2)}p,
    \]
    and hence
    \[\log(b') \le \log(p) - \log(D\log(A_2)) + \log(1.001) \le \log(p) - \log(D\log(A_2)) + 0.001.
    \]
    If $D=2$, combining this with \eqref{eqn:ubqadiclogs} directly gives the desired result. If $D=1$, the result follows from the observation that 
    \begin{multline*}\max\left\{ \log(b') + \log(\log(q)) + 0.4, 10\log(q)\right\}^2
    \le \\ 4\max\left\{ \log(b') + \log(\log(q)) + 0.4, 5\log(q)\right\}^2.
    \end{multline*}
\end{proof}

An immediate application of this Lemma is the following Corollary, which gives a lower bound for $y^p$ in terms of $c, q$ and $k$. We will need this in Section \ref{Subsec:complexlogs} to bound $p$. 

\begin{corollary}
    \label{cor:ypbig}
    Suppose that $(x, y, \alpha, p)$ is a solution to \eqref{eqn:main} where $y$ is even and $p > 3\cdot 10^7$. Let $c = C_1q$ if $\alpha = 2k+1$ and $c = C_1$ if $\alpha = 2k$. Then
    \[y^p > 100cq^{2k}.\]
\end{corollary}

\begin{proof}
    Suppose for contradiction that $y^p \le 100cq^{2k}$. Then, taking logarithms yields that
    \[p\log(y) \le \log(100) + \log(c) + 2k\log(q) \le 2\log(10) + \log(C_1) + (2k+1)\log(q),
    \]
    because $c \le C_1q$. Then, by Lemma \ref{lemma:qadiclogs}, we have that 
    \begin{multline*}
        p \le \frac{2\log(10)}{\log(y)} + \frac{\log(C_1)}{\log(y)} + \frac{\log(q)}{\log(y)} +  
        \frac{96qs}{\log^3(q)}\frac{q^f-1}{q-1}D^2 \log(A_2) \cdot \\ \max\{ \log(p) + \log(\log(q)) -\log(D\log(A_2))+0.401, 5\log(q)\}^2\log(y)
    \end{multline*}
    Using the lower bounds for $y$ given in Lemma \ref{lemma:boundy}, we show that, for all the pairs in \eqref{eqn:badpairsqodd} and \eqref{eqn:badpairsqeven}, $p \le 10^7$. This is a contradiction with the assumption that $p > 3\cdot 10^7$, and so $y^p > 100cq^{2k}$.
\end{proof}

\subsection{Linear forms in complex logarithms}
\label{Subsec:complexlogs}
Before applying the results on lower bounds on linear forms on three complex logarithms available in \cite{Voutier}, we need to define the linear forms in logarithms that we shall be considering and find an upper bound for it. Using Lemma \ref{lemma:factlfl}, we may do so now. Let $\Delta_2$ be given by
\begin{equation}
    \label{eqn:Delta2}
\Delta_2 = s\log\left(\frac{C_1x-q^k\sqrt{-c}}{C_1x+q^k\sqrt{-c}}\right) = p\log(\varepsilon_1 \gamma) + \log(\varepsilon_2 \beta) + j\pi i,
\end{equation}
where we consider the principal branches of the logarithm and $\varepsilon_2 \in \{\pm 1\}$ is chosen such that $|\log(\varepsilon_2 \beta)| < \pi/2$, $\varepsilon_1 \in \{\pm 1\}$ is chosen such that $\log(\varepsilon_1 \gamma)$ and $\log(\varepsilon_2 \beta)$ have opposite signs and $j$ is such that the quantity $|\Delta_2|$ is minimal. With this, we are able to prove the following Lemma.

\begin{lemma}
    \label{lemma:upperboundlogDelta2}
    Let $(x, y, p, \alpha)$ be a solution to \eqref{eqn:main}, with $y$ even and $p > 3\cdot 10^7$. Then, either:

    \begin{itemize}
        \item $\alpha$ is odd and 
    \[\log(|\Delta_2|) \le
    \begin{cases}
        -0.49p\log(y) + 385.38\log(y) + 1.79 & \text{ if } (C_1, q) = (1, 7), \\
        -0.48p\log(y) + 1264.35\log(y) + 3.48 &
        \text{ if } (C_1, q) = (1, 23), \\
        -0.49p\log(y) + 486\log(y) + 2.17 & \text{ if } (C_1, q) = (3, 5), \\
        -0.49p\log(y) + 718.20\log(y) + 3.34 & \text{ if } (C_1, q) = (3, 13), \\
        -0.49p\log(y) + 784.34\log(y) + 2.17 & \text{ if } (C_1, q) = (5, 3), \\
        -0.49p\log(y) + 670\log(y) + 3.51 & \text{ if } (C_1, q) = (5, 11), \\
        -0.49p\log(y) + 469.25\log(y) + 3.51 & \text{ if } (C_1, q) = (11, 5), \\
        -0.48p\log(y) + 272.28\log(y) + 3.34 & \text{ if } (C_1, q) = (13, 3), \\
        -0.47p\log(y) + 1003.74\log(y) + 4.91 & \text{ if } (C_1, q) = (13, 11), \\
    \end{cases}\]
    \item or $\alpha$ is even and
    \[\log(|\Delta_2|) \le
    \begin{cases}
        -0.49p\log(y) + 3900.88\log(y) + 1.79 & \text{ if } (C_1, q) = (7, 3), \\
        -0.49p\log(y) + 1458\log(y) + 1.79 & \text{ if } (C_1, q) = (7, 5), \\
        -0.49p\log(y) + 182.46\log(y) + 1.79 & \text{ if } (C_1, q) = (7, 11), \\
        -0.48p\log(y) + 2613.88\log(y) + 1.79 & \text{ if } (C_1, q) = (7, 13), \\
        -0.49p\log(y) + 223.36\log(y) + 1.79 & \text{ if } (C_1, q) = (7, 23), \\
        -0.49p\log(y) + 1541.51\log(y) + 2.165 & \text{ if } (C_1, q) = (15, 7), \\
        -0.49p\log(y) + 364.92\log(y) + 2.165 & \text{ if } (C_1, q) = (15, 11), \\
        -0.48p\log(y) + 2189.48\log(y) + 2.165 & \text{ if } (C_1, q) = (15, 17), \\
    \end{cases}\]
    
    \end{itemize}

\end{lemma}

\begin{proof}
    Let us begin by defining $\Delta'$ and $\Delta$ by the following expressions.
    \[\Delta' = \frac{C_1x-q^k\sqrt{-c}}{C_1x+q^k\sqrt{-c}}, \quad \Delta = \log\left(\Delta'\right).    \]
    We note that
    \[|\Delta'-1| = \left|\frac{C_1x-q^k\sqrt{-c}}{C_1x+q^k\sqrt{-c}} - 1 \right| = \frac{2q^k\sqrt{c}}{|C_1x+q^k\sqrt{-c}|} = \frac{2q^k\sqrt{c}}{y^{p/2}}.
    \]
    By Corollary \ref{cor:ypbig}, we have that $y^{p/2} > 10q^k\sqrt{c}$ and, consequently,
    \[\left|\Delta' - 1 \right| < \frac{1}{5}.\]
    Then, applying Lemma B.2 in \cite{Smart} gives that
    \[|\Delta| = |\log(\Delta')| \le -10\log\left(\frac{4}{5}\right)\frac{q^k\sqrt{c}}{y^{p/2}}.
    \]
    If we then take logarithms and replace $c$ by its definition in \eqref{eqn:cdalphaodd} if $\alpha$ is odd and by its definition in \eqref{eqn:cdalphaeven} if $\alpha$ is even, it follows that
    \[\log(|\Delta|) \le 
    \begin{cases}
        0.81 + \left(k+\frac{1}{2}\right)\log(q) + \frac{1}{2}\log(C_1) - \frac{p}{2}\log(y) & \text{ if } \alpha \text{ is odd.} \\
        0.81 + k\log(q) + \frac{1}{2}\log(C_1) - \frac{p}{2}\log(y) & \text{ if } \alpha \text{ is even.} \\
    \end{cases}
    \]
    Now, the definition of $\Delta_2$ in \eqref{eqn:Delta2} gives that $\log(|\Delta_2|) = \log(s) + \log(\Delta)$. The result then follows for each case by considering the upper bound for $k$ given in Lemma \ref{lemma:qadiclogs}, as well as the appropiate values for $C_1$, $q$, $f$, $D$, $s$ and $A_2$.
\end{proof}

The final ingredient that we will need before finishing the proof of Proposition \ref{prop:lflub} is an upper bound for $|j|$. This follows from the definition of $\Delta_2$ and is the content of the following Lemma.

\begin{lemma}
    \label{lemma:upperboundj}
    Suppose that $(x, y, \alpha, p)$ is a solution to \eqref{eqn:main} with $y$ even and $p > 3\cdot 10^7$, and let $\Delta_2$ be as in \eqref{eqn:Delta2}. Then,
    \[|j| \le p.\]
\end{lemma}
\begin{proof}
    By definition of $\Delta_2$, along with the triangle inequality, we have that
    \[|j|\pi \le |\Delta_2| + |p\log(\varepsilon_1 \gamma) + \log(\varepsilon_2 \beta)| < \frac{\pi}{2} + p\pi = \left(p + \frac{1}{2}\right)\pi,
    \]
    where the last inequality follows because $|\Delta_2| \le \pi/2$ and due to the fact that $\log(\varepsilon_1 \gamma)$ and $\log(\varepsilon_2 \beta)$ have opposite signs. From here, it readily follows that $|j| \le p$.
\end{proof}
With this, we are finally able to apply the techniques in \cite{Voutier} in order to finish the proof of Proposition \ref{prop:lflub}.

\begin{proof}(of Proposition \ref{prop:lflub})
    We will use the publicly available PARI/GP \cite{Pari} code associated to \cite{Voutier}, which will allow us to find an upper bound for $p$.

    We note that this code makes use of Matveev's theorem (Theorem 2.1 in \cite{Voutier}, originally presented in \cite{Matveev}), in order to obtain an initial upper bound for $p$. Then, it exploits the improved lower bounds for linear forms in three logarithms given in Theorem 4.1 of \cite{Voutier} to iteratively improve upon this upper bound of $p$, obtaining the final values of $N_0(C_1, q)$ given in \eqref{eqn:N0}. The correctness of this computation can be checked with the PARI/GP code available in the GitHub repository. 
    
    The following are the necessary input parameters which are common for all $(C_1, q)$, in the notation of Theorem 4.1 in \cite{Voutier}:
    \[b_1 = p, \quad b_2 = 1, \quad  b_3 = j, \quad \alpha_1 = \varepsilon_1\gamma, \quad  \alpha_2 = \varepsilon_2\beta, \quad \alpha_3 = -1.\]
    \[\mathcal{D} = \frac{\Q[\alpha_1, \alpha_2, \alpha_3]}{\R[\alpha_1, \alpha_2, \alpha_3]} = \frac{2}{2} = 1.\]
    As shown in the proof of Lemma \ref{lemma:qadiclogs}, the heights of the $\alpha_i$, which we also need for this computation, are given by
    \[h(\alpha_1) = \frac{s}{2}\log(y), \quad h(\alpha_2) = s\log(2), \quad h(\alpha_3) = 0.\]
    Also, $|j| \le p$ by Lemma \ref{lemma:upperboundj} and $\log(|\Delta_2|)$ is bounded above by the quantities in Lemma \ref{lemma:upperboundlogDelta2}. Finally, in the notation of Matveev's theorem (Theorem 2.1 in \cite{Voutier}), we also have:
    \[D = \chi = 2.\]
    With this, we perform three iterations of the code for each pair $(C_1, q)$. For the interested reader, we remark that the parameters $L, m, \chi$ and $\rho$ obtained in each iteration (see Section 5.2 of \cite{Voutier}), as well as the estimate on $p$ after each iteration, are recorded on Table \ref{tab:lflparametersalphaodd} for $\alpha$ odd and in Table \ref{tab:lflparametersalphaeven} for $\alpha$ even. This finishes the proof of the Proposition.

    \begin{table}[!ht]
        \begin{center}
            \begin{tabular}{||c|c|c|c|c|c||}
                \hline 
                 $(C_1, q)$ & $L$  & $m$ & $\rho$ & $\chi$ & Upper bound on $p$ \\
                 \hline\hline 
                 $(1, 7)$ & $115$ & $12.50$ & $5.80$ & $0.044$ & $2.089874 \cdot 10^8$ \\
                 $(1, 7)$ & $75$ & $14.60$ & $5.30$ & $0.080$ & $7.979286 \cdot 10^7$ \\
                 $(1, 7)$ & $72$ & $13.60$ & $5.40$ & $0.080$ & $7.234157 \cdot 10^7$ \\
                \hline
                $(1, 23)$ & $106$ & $9.00$ & $7.40$ & $0.072$ & $4.524352 \cdot 10^8$ \\
                 $(1, 23)$ & $67$ & $9.80$ & $6.90$ & $0.100$ & $1.663534 \cdot 10^8$ \\
                 $(1, 23)$ & $63$ & $9.75$ & $7.00$ & $0.102$ & $1.514725 \cdot 10^8$ \\
                \hline
                 $(3, 5)$ & $102$ & $16.50$ & $6.20$ & $0.076$ & $1.151876 \cdot 10^8$ \\
                 $(3, 5)$ & $61$ & $16.40$ & $6.10$ & $0.100$ & $3.915560 \cdot 10^7$ \\
                 $(3, 5)$ & $57$ & $17.45$ & $6.00$ & $0.102$ & $3.476178 \cdot 10^7$ \\
                 \hline  
                 $(3, 13)$ & $118$ & $9.00$ & $6.60$ & $0.052$ & $3.641642 \cdot 10^8$ \\
                 $(3, 13)$ & $74$ & $11.60$ & $5.90$ & $0.080$ & $1.372399 \cdot 10^8$ \\
                 $(3, 13)$ & $65$ & $11.95$ & $6.20$ & $0.080$ & $1.243438 \cdot 10^8$ \\
                 \hline
                 $(5, 3)$ & $102$ & $16.50$ & $6.20$ & $0.076$ & $1.151876 \cdot 10^8$ \\
                 $(5, 3)$ & $61$ & $16.40$ & $6.10$ & $0.100$ & $3.915560 \cdot 10^7$ \\
                 $(5, 3)$ & $57$ & $17.45$ & $6.00$ & $0.102$ & $3.476178 \cdot 10^7$ \\
                 \hline
                 $(5, 11)$ & $102$ & $10.50$ & $7.20$ & $0.072$ & $2.659731 \cdot 10^8$ \\
                 $(5, 11)$ & $67$ & $11.00$ & $6.50$ & $0.094$ & $9.270785 \cdot 10^7$ \\
                 $(5, 11)$ & $62$ & $10.30$ & $6.80$ & $0.098$ & $8.334595 \cdot 10^7$ \\
                 \hline
                 $(11, 5)$ & $102$ & $10.50$ & $7.20$ & $0.072$ & $2.659731 \cdot 10^8$ \\
                 $(11, 5)$ & $67$ & $11.00$ & $6.50$ & $0.094$ & $9.270785 \cdot 10^7$ \\
                 $(11, 5)$ & $62$ & $10.30$ & $6.80$ & $0.098$ & $8.334595 \cdot 10^7$ \\
                 \hline
                 $(13, 3)$ & $118$ & $9.00$ & $6.60$ & $0.052$ & $3.71751 \cdot 10^8$ \\
                 $(13, 3)$ & $67$ & $11.00$ & $6.50$ & $0.08$ & $1.40676 \cdot 10^8$ \\
                 $(13, 3)$ & $68$ & $11.95$ & $6.00$ & $0.080$ & $1.273969 \cdot 10^8$ \\
                 \hline
                 $(13, 11)$ & $112$ & $7.00$ & $8.00$ & $0.074$ & $1.020209 \cdot 10^9$ \\
                 $(13, 11)$ & $65$ & $8.00$ & $7.90$ & $0.108$ & $3.816492 \cdot 10^8$ \\
                 $(13, 11)$ & $65$ & $7.55$ & $7.80$ & $0.110$ & $3.499196 \cdot 10^8$ \\
                 \hline
             \end{tabular}
             \caption{Parameters obtained in each iteration of the code associated to \cite{Voutier} for $\alpha$ odd.}
            \label{tab:lflparametersalphaodd}
        \end{center}
    \end{table}

    \begin{table}[!ht]
        \begin{center}
            \begin{tabular}{||c|c|c|c|c|c||}
                \hline 
                 $(C_1, q)$ & $L$  & $m$ & $\rho$ & $\chi$ & Upper bound on $p$ \\
                 \hline\hline 
                 $(7, 3)$ & $115$ & $12.50$ & $5.80$ & $0.044$ & $2.089874 \cdot 10^8$ \\
                 $(7, 3)$ & $74$ & $14.60$ & $5.30$ & $0.08$ & $7.979286 \cdot 10^7$ \\
                 $(7, 3)$ & $72$ & $13.60$ & $5.40$ & $0.08$ & $7.234157 \cdot 10^7$ \\
                \hline
                 $(7, 5)$ & $116$ & $13.25$ & $5.65$ & $0.044$ & $2.087874 \cdot 10^8$ \\
                 $(7, 5)$ & $78$ & $14.00$ & $5.20$ & $0.056$ & $7.828204 \cdot 10^7$ \\
                 $(7, 5)$ & $75$ & $14.50$ & $5.10$ & $0.056$ & $7.083124 \cdot 10^7$ \\
                 \hline  
                 $(7, 11)$ & $116$ & $13.25$ & $5.65$ & $0.044$ & $2.087874 \cdot 10^8$ \\
                 $(7, 11)$ & $78$ & $14.00$ & $5.20$ & $0.056$ & $7.828204 \cdot 10^7$ \\
                 $(7, 11)$ & $75$ & $14.50$ & $5.10$ & $0.056$ & $7.083124 \cdot 10^7$ \\
                 \hline
                 $(7, 13)$ & $116$ & $13.25$ & $5.65$ & $0.044$ & $2.131371 \cdot 10^8$ \\
                 $(7, 13)$ & $74$ & $14.75$ & $5.30$ & $0.056$ & $8.011225 \cdot 10^7$ \\
                 $(7, 13)$ & $70$ & $14.10$ & $5.40$ & $0.056$ & $7.236925 \cdot 10^7$ \\
                 \hline

                 $(7, 23)$ & $116$ & $13.25$ & $5.65$ & $0.044$ & $2.087874 \cdot 10^8$ \\
                 $(7, 23)$ & $78$ & $14.00$ & $5.20$ & $0.056$ & $7.828204 \cdot 10^7$ \\
                 $(7, 23)$ & $75$ & $14.50$ & $5.10$ & $0.056$ & $7.083124 \cdot 10^7$ \\
                 \hline
                 $(15, 7)$ & $109$ & $14.75$ & $6.15$ & $0.076$ & $1.149974 \cdot 10^8$ \\
                 $(15, 7)$ & $61$ & $15.00$ & $6.30$ & $0.100$ & $3.913906 \cdot 10^7$ \\
                 $(15, 7)$ & $56$ & $16.50$ & $6.20$ & $0.104$ & $3.472013 \cdot 10^7$ \\
                 \hline
                 $(15, 11)$ & $109$ & $14.75$ & $6.15$ & $0.076$ & $1.149974 \cdot 10^8$ \\
                 $(15, 11)$ & $61$ & $15.00$ & $6.30$ & $0.100$ & $3.913906 \cdot 10^7$ \\
                 $(15, 11)$ & $56$ & $16.50$ & $6.20$ & $0.104$ & $3.472013 \cdot 10^7$ \\
                 \hline
                 $(15, 17)$ & $103$ & $15.50$ & $6.30$ & $0.076$ & $1.177119 \cdot 10^8$ \\
                 $(15, 17)$ & $60$ & $16.25$ & $6.20$ & $0.100$ & $4.007117 \cdot 10^7$ \\
                 $(15, 17)$ & $58$ & $16.10$ & $6.10$ & $0.102$ & $3.547538 \cdot 10^7$ \\
                 \hline
             \end{tabular}
             
          \caption{Parameters obtained in each iteration of the code associated to \cite{Voutier} for $\alpha$ even.}
          \label{tab:lflparametersalphaeven}
        \end{center}
    \end{table}
    
\end{proof}

\begin{remark}
    We note that the improved lower bounds available in \cite{Voutier} allow for a much better bound than that present in previous literature. For comparison, our values for $N_0(1, 7)$ and $N_0(1, 23)$ are between  $50\%$ and $70\%$ smaller than the corresponding values in Proposition 15.2 of \cite{BennettSiksek2}, allowing for a much more efficient computation.
\end{remark}

\section{Conclusions}
\label{Sec:conclusions}
We compile all the work in previous sections to finish the proof of Theorem \ref{thm:main}.

\begin{proof}(of Theorem \ref{thm:main})
After Sections \ref{Sec:smallexponents}, \ref{Sec:yodd} and \ref{Sec:ThueMahleryEven}, we are left with the case of \eqref{eqn:main} where $y$ is even and $p \ge 11$ is prime. Then, Propositions \ref{prop:nobigsolutions1}, \ref{prop:nobigsolutions2} and \ref{prop:lflub} show that the only solutions with $p \ge 11$ are those corresponding to the tuples in \eqref{eqn:missingsolutions}. All of the solutions are in Tables \ref{tab:solutions1} and \ref{tab:solutions2}, thereby finishing the proof.
\end{proof}

In theory, there would be no reason to restrict our analysis to $1 \le C_1 \le 20$ and $2 \le q < 25$. However, we remark that there are some solutions to \eqref{eqn:main} with $C_1 = 21$ and exponent $p = 17$, such as the following
\[21\cdot 79^2 + 11^1 = 2^{17}.\]
The same is true for the following value of $q$  that we would need to consider if we were to extend the ranges ($q = 29$), since the following identity holds
\[3\cdot209^2+29^1=2^{17}.\]
Since there is a solution in both cases, all the techniques that we have developed in Section \ref{Sec:specificexponents} would fail, forcing us to solve a Thue--Mahler equation of degree $p = 17$. 

As of now, and without significant computational improvements, this is im\-po\-ssi\-ble to do. Therefore, extending the solution of \eqref{eqn:main} to bigger ranges is probably impossible unless significantly new techniques are introduced.


\end{document}